\documentclass[]{siamltex}
\usepackage{subfigure}

\usepackage{amssymb,amsmath,newlfont,enumerate}

\usepackage{graphicx,epsfig,color}
\usepackage{latexsym}
\usepackage{bm}
\usepackage{float}
\usepackage{multirow}

\usepackage[english]{babel}

\def\paragraph#1{{\bf #1\ }}

\def\pd#1#2{\dfrac{\partial #1}{\partial #2}}

\newtheorem{remark}[section]{Remark}

\def\MM{\mathbb{M}}

\def\RR{\mathbb{R}}

\def\AA{\mathcal{A}}

\def\VV{\mathcal{V}}
\def\GG{\mathcal{G}}

\let\ds\displaystyle
\let\eps\varepsilon

\title{Anisotropic finite elements with high aspect ratio for an
  Asymptotic Preserving method for highly anisotropic elliptic
  equations}

\author{Jacek Narski\thanks{Universit\'e de Toulouse, UPS, Institut de
  Math\'ematiques de Toulouse, F-31062 Toulouse, France}}

\begin{document}

\maketitle

\begin{abstract}
  The concern of this work is the generalization of an
  Asymptotic Preserving method for the highly anisotropic elliptic
  equations presented in \cite{DLNN}. The limitations of the method
  introduced there in are omitted by the introduction of a
  stabilization term in the Asymptotic Reformulation. Furthermore,
  anisotropic error indicators and mesh adaptation algorithms are
  proposed and tested allowing to reduce considerably the number of
  mesh points required to achieve prescribed precision. Reported
  meshes have maximum aspect ratio greater than 500.
\end{abstract}



%%%%%%%%%%%%%%%%%%%%%%%
\section{Introduction}
%%%%%%%%%%%%%%%%%%%%%%%

Anisotropic problems are common in mathematical modeling of physical
problems. They appear in various fields of application, such as flows
in porous media \cite{porous1,TomHou}, semiconductor modeling
\cite{semicond}, quasi-neutral plasma simulations \cite{Navoret},
image processing \cite{image1,Weickert}, atmospheric or oceanic flows
\cite{ocean} and so on, the list being not exhaustive. The direct
motivation of this work is related to numerical simulations of
strongly magnetized plasma such as internal fusion plasma of tokamak
\cite{Beer,Sangam}, atmospheric plasma \cite{Kelley2,Kes_Oss} or
plasma thrusters \cite{SPT}. In this context a strong magnetic field
is defining the anisotropy direction. Fast rotation of charged
particles around magnetic field lines is causing a large number of
collisions in the plane perpendicular to the magnetic field. On the
other hand the motion in the direction of the field is rather
undisturbed. In consequence the particle mobility depends on the
direction and may differ by several orders of magnitude. Anisotropy
ratio $1/\varepsilon $ can be as high as $10^{10}$. 

The main difficulty associated with these anisotropic problems is that
they are singular in the limit $\varepsilon \rightarrow 0$. On the
discrete level this is manifested by very bad conditioning of linear
systems obtained by a direct discretization of the problem for
$\varepsilon \ll 1$. 
%% Various methods has been developed to deal with this problem. For
%% example the symmetric and asymmetric difference schemes were investigated in
%% \cite{Gunter05,Sharma07}. The use of slope limiters is advocated in
%% \cite{Kuz09}. The $hp$-finite element method is known to get good
%% results for singular perturbation problems \cite{Melenek}.
In this paper we propose an approach based on
the Asymptotic Preserving reformulation introduced initially by Shi
Jin in \cite{ShiJin}. Our approach is an extension of the method
proposed in a previous paper \cite{DLNN} to the case of more general
anisotropy field structure (such as closed field lines).

The model problem we are interested in, reads
\begin{gather}
  \left\{ 
    \begin{array}{ll}
      - \nabla \cdot \mathbb A_\varepsilon  \nabla u^{\varepsilon } = f       & \text{ in }
      \Omega, \\[3mm]
      n \cdot \mathbb A_\varepsilon  \nabla u^{\varepsilon }= 0
      & \text{ on } \Gamma_N\,,\\[3mm]
      u^{\varepsilon }= 0
      & \text{ on } \Gamma_D\,,
    \end{array}
  \right.
  \label{InitP}
\end{gather}
where $\Omega \subset \RR^{2}$ is a
bounded domain with boundary $\partial \Omega = \Gamma_D \cup
\Gamma_N$ and outward normal $n$. The direction of the anisotropy is
given by a vector field $B$, where we suppose $\text{div} B = 0$ and
$B \neq 0$. The direction of $B$ shall be denoted by the unit vector field
$b = B/|B|$. The domain boundary is decomposed into $\Gamma_D:= \{ x
\in \partial\Omega \ | \ b (x) \cdot n = 0 \}$
and $\Gamma_N:= \partial \Omega \backslash \Gamma_D$.  The
anisotropic diffusion matrix is then given by
\begin{gather} 
  \mathbb A_\varepsilon  = \frac{1}{\varepsilon }
  A_\parallel b \otimes b + (Id - b \otimes b)A_\perp (Id - b \otimes
  b)\,.
  \label{eq:Jh0a}
\end{gather}
The scalar field $A_\parallel>0$ and the symmetric positive definite
matrix field $A_\perp$ are of order one while the parameter $0 <
\varepsilon < 1$ can be very small, provoking thus the high anisotropy
of the problem.  The system becomes ill posed if we consider the
formal limit $\eps \rightarrow 0$. It is thus very ill conditioned for
$\eps \ll 1$. 
%% The goal of the present paper is to circumvent this
%% difficulty and to propose a numerical scheme which is uniformly
%% convergent with respect to the parameter $\eps$.

This problem has been studied before in the Asymptotic Preserving
context. A special case of anisotropy direction aligned with one of
the coordinate axis was addressed in \cite{DDN}. A generalization of
this approach was presented in \cite{besse}, where the problem with
curvilinear anisotropy field was reduced to one with the anisotropy
direction aligned with the coordinate system by a change of
variables. Another work \cite{DDLNN} proposed a different
generalization based rather on the introduction of Lagrange
multipliers. This resulted in a considerably bigger linear system but
allowed to avoid a necessity of change of variables which could be
troublesome for time dependent anisotropy direction. Finally, a
different method presented in \cite{DLNN} allowed to reduce
considerably computational cost without any adaptation of the
coordinate system. All those methods shared the same drawback: they
didn't allow more complex geometries such as the presence of closed
field lines.

In this paper we introduce yet another Asymptotic Preserving scheme,
improving the idea presented in \cite{DLNN} and removing the
restrictions on the anisotropy direction by a simple penalty
stabilization technique. Furthermore, the anisotropic error indicator
is presented and the mesh adaptation algorithm developed in order to
optimize the number of mesh points required to obtain a prescribed error.

The outline of the paper is following. Section \ref{sec:prob_def}
contains a definition of the problem and introduces the Asymptotic
Preserving reformulation. Section \ref{sec:num_met} describes an
anisotropic error indicator and mesh adaptation algorithm. They are
both tested and the numerical results are provided.

%%%%%%%%%%%%%%%%%%%%%%%
\section{Problem definition}\label{sec:prob_def}
%%%%%%%%%%%%%%%%%%%%%%%
We consider a two dimensional anisotropic problem, given on a regular,
bounded domain $\Omega \subset \mathbb R^2$, with boundary $\partial
\Omega$. The direction of the anisotropy is defined by the vector
field $b(x)$, which satisfies the following hypothesis\\

\noindent {\bf Hypothesis A} {\it The field $b(x)$ is derived from a
  vector field $B(x)= |B(x)|\, b(x)$, satisfying $div\,\, B(x)=0$ and
  $|b(x)|=1$ for all $x \in \Omega$. Moreover, we suppose that $b \in
  (C^{\infty}(\Omega))^d$.}\\
  %% and that the field lines of $b$ have finite
  %% length in the domain $\Omega$ with an initial and an end point.}\\

\noindent Given this vector field $b$, one can decompose now vectors
$v \in \mathbb R^2$, gradients $\nabla \phi$, with $\phi(x)$ a scalar
function, and divergences $\nabla \cdot v$, with $v(x)$ a vector
field, into a part parallel to the anisotropy direction and a part
perpendicular to it.  These parts are defined as follows :
\begin{equation} 
  \begin{array}{llll}
    \ds v_{||}:= (v \cdot b) b \,, & \ds v_{\perp}:= (Id- b \otimes b) v\,, &\textrm{such that}&\ds
    v=v_{||}+v_{\perp}\,,\\[3mm]
    \ds \nabla_{||} \phi:= (b \cdot \nabla \phi) b \,, & \ds
    \nabla_{\perp} \phi:= (Id- b \otimes b) \nabla \phi\,, &\textrm{such that}&\ds
    \nabla \phi=\nabla_{||}\phi+\nabla_{\perp}\phi\,,\\[3mm]
    \ds \nabla_{||} \cdot v:= \nabla \cdot v_{||}  \,, & \ds
    \nabla_{\perp} \cdot v:= \nabla \cdot v_{\perp}\,, &\textrm{such that}&\ds
    \nabla \cdot v=\nabla_{||}\cdot v+\nabla_{\perp}\cdot v\,,
  \end{array}
\end{equation}
where we denoted by $\otimes$ the vector tensor product. With these
notations we can now introduce the mathematical problem, the so-called
Singular Perturbation problem, whose numerical resolution is the main
concern of this paper.

%%%%%%%%%%%%%%%%%%%%%%%
\subsection{The Singular Perturbation problem (P-model)}
%%%%%%%%%%%%%%%%%%%%%%%

The objective of this paper is to introduce an efficient scheme for
the precise ($\varepsilon$-independent) resolution of the following
Singular Perturbation problem
\begin{gather}
  (P)\,\,\,
  \left\{
    \begin{array}{ll}
      -{1 \over \varepsilon} \nabla_\parallel \cdot 
      \left(A_\parallel \nabla_\parallel \phi^{\varepsilon }\right) 
      - \nabla_\perp \cdot 
      \left(A_\perp \nabla_\perp \phi^{\varepsilon }\right) 
      = f 
      & \text{ in } \Omega, \\[3mm]
      {1 \over \varepsilon} 
      n_\parallel \cdot 
      \left( A_\parallel \nabla_\parallel \phi^{\varepsilon } \right)
      +
      n_\perp \cdot 
      \left(A_\perp \nabla_\perp \phi^{\varepsilon }\right) 
      = 0
      & \text{ on } \partial\Omega _{in} \cup \partial\Omega _{out},  \\[3mm]
      \phi^{\varepsilon }= 0
      & \text{ on } \partial\Omega _D\,,
    \end{array}
  \right.
  \label{eq:J07a} 
\end{gather}
where $n$ is the outward normal to $\Omega $ and the boundaries are defined by 
\begin{gather}
  \partial\Omega _D = \{ x \in \partial\Omega \ | \ b (x) \cdot n = 0 \},\\
  \partial\Omega_{in} = \{ x \in \partial\Omega \ | \ b (x) \cdot n < 0 \},\\
  \partial\Omega_{out} = \{ x \in \partial\Omega \ | \ b (x) \cdot n > 0 \}
  \label{eq:Ju9a}.
\end{gather}
The parameter $0<\eps <1$ can be very small and is responsible for the
high anisotropy of the problem. We shall assume in the rest of this
paper the following hypothesis on the diffusion and source terms\\

\noindent {\bf Hypothesis B} \label{hypo} {\it
  Let $f \in L^2(\Omega)$ and $\overset{\circ}{\partial\Omega _D} \neq \varnothing$.
  Furthermore, the diffusion coefficients $A_{\parallel} \in
  L^{\infty} (\Omega)$ and $A_{\perp} \in \MM_{d \times d} (L^{\infty}
  (\Omega))$ are supposed to satisfy
  \begin{gather}
        0<A_0 \le A_{\parallel}(x) \le A_1\,, \quad  \textrm{f.a.a}\,\,\,x \in \Omega,
        \label{eq:J48a1}
        \\[3mm]
        A_{\perp} (x) b(x)=A_{\perp}^t (x) b(x)=0\,, \quad  \textrm{f.a.a}\,\,\, x \in \Omega,
        \label{eq:J48a2}
        \\[3mm]
        A_0 ||v||^2 \le v^t A_{\perp}(x) v \le A_1 ||v||^2\,,
        \quad \forall v\in \mathbb R^d\,\,\, \text{with} \,\,\, v\cdot
        b(x)=0\,\,\, \text{and} \,\,\,  \textrm{f.a.a}\,\,\, x \in \Omega.
        \label{eq:J48a3}
  \end{gather}
}

\noindent As we conceive to use the finite element method for the
numerical resolution of the P-problem, let us put (\ref{eq:J07a})
under variational form. For this let $\mathcal{V}$ be the Hilbert space
$$
\mathcal{V}:=\{ \phi \in H^1(\Omega)\,\, / \,\, \phi_{| \partial \Omega_D} =0 \}
\,, \quad (\phi,\psi)_{\mathcal{V}} := (\nabla_{\parallel} \phi,\nabla_{\parallel}
\psi)_{L^2} + \eps (\nabla_{\perp} \phi,\nabla_{\perp}
\psi)_{L^2}\,.
$$
We are searching thus for $\phi^\eps \in \mathcal{V}$, solution of
\begin{gather}
  \label{eq:Ja8a} 
  a_\parallel(\phi^\eps,\psi) + \eps
  a_\perp(\phi^\eps,\psi)=\eps (f,\psi)\,, \quad \forall \psi \in \mathcal{V}\,, 
\end{gather} 
where $(\cdot,\cdot)$ stands for the standard $L^2$ scalar product and
the continuous, bilinear forms $a_\parallel : \mathcal{V} \times \mathcal{V}
\rightarrow \RR$ and $a_\perp: \mathcal{V} \times \mathcal{V} \rightarrow \RR$
are given by
\begin{gather} \label{bil}
  \begin{array}{lll}
    \ds a_\parallel (\phi,\psi)&:=&\ds \int_{\Omega} A_{||} \nabla_{||}
    \phi \cdot \nabla_{||}\psi\, dx\,, \quad  a_\perp(\phi,\psi):=\ds
    \int_{\Omega} ( A_{\perp} \nabla_{\perp}
    \phi) \cdot \nabla_{\perp}\psi\, dx\,.
  \end{array}
\end{gather}
Thanks to Hypothesis B and the Lax-Milgram theorem, the problem
(\ref{eq:J07a}) admits a unique solution $\phi^\eps \in \mathcal{V}$ for
all fixed $\eps >0$.
However, the numerical resolution of (\ref{eq:J07a}) is very inadequate for
$\varepsilon \ll 1$. When $\varepsilon$ tends to zero, the problem
reduces to
\begin{gather}
  \left\{
    \begin{array}{ll}
      \displaystyle
      -\nabla_\parallel \cdot 
      \left(A_\parallel \nabla_\parallel \phi\right) 
      = 0 
      & \text{ in } \Omega, \\[3mm]
      \displaystyle
      n_\parallel \cdot 
      \left( A_\parallel \nabla_\parallel \phi \right)
      = 0
      & \text{ on } \partial\Omega _{in} \cup \partial\Omega _{out},  \\[3mm]
      \displaystyle
      \phi^{0 }= 0
      & \text{ on } \partial\Omega _D.
    \end{array}
  \right.
  \label{eq:Jc8a}
\end{gather}
This is an ill-posed problem as it has an infinite number of solutions
$\phi \in \mathcal G$, where
\begin{gather}
  \mathcal G = \{ \phi \in \mathcal V \ | \ \nabla_\parallel \phi =0\}\,,
  \label{eq:Jd8a}
\end{gather}
is the Hilbert space of functions, which are constant along the field
lines of $b$. On the discrete level this is manifested by a very bad
conditioning of the system for small values of $\varepsilon $.
However, as shown in \cite{DDLNN}, 
the solution $\phi ^\varepsilon \in \mathcal V$
converges to $\phi ^0 \in \mathcal G$, a unique solution of
\begin{gather}
  (L)\,\,\,
  \int_\Omega  A_{\perp} \nabla_{\perp} \phi^{0}
  \cdot 
  \nabla_{\perp} \psi \, dx
  =
  \int_\Omega f \psi \, dx
  \;\; , \;\;  \forall \psi \in \mathcal G\,.
  \label{eq:Jv9a}
\end{gather}

%%%%%%%%%%%%%%%%%%%%%%%
\subsection{The Asymptotic Preserving approach (AP-model)}
%%%%%%%%%%%%%%%%%%%%%%%

Let us introduce a so called AP-formulation, which is a
reformulation of the Singular Perturbation problem (\ref{eq:J07a}),
permitting a ``continuous'' transition from the (P)-problem
(\ref{eq:J07a}) to the (L)-problem (\ref{eq:Jv9a}), as $\eps
\rightarrow 0$. The AP-formulation was introduced and is a subject of
more detailed analysis in a separate publication \cite{DLNN}. We will
shortly recall the results of the previous studies. For this, each
function shall be decomposed into two parts: constant part along the
anisotropy direction and a part containing fluctuations. The constant
part converges to the limit solution and the fluctuating to $0$ as $\varepsilon
\rightarrow 0$ (see also \cite{DLNN}).

Let us introduce the following Hilbert space:
\begin{gather}
  \mathcal A : = 
  \{ q \in L^2(\Omega ) \ / \nabla_\parallel q \in L^2(\Omega )
  \text{ and } q|_{\partial \Omega_{in}} =0 \}
  \\
  (q,w)_{\mathcal A} = (\nabla_\parallel q, \nabla_\parallel w) \;\; ,
  \;\;  \forall q,w\in \mathcal A
  \label{eq:Jg8a}.
\end{gather}

Let $\phi^{\eps}$ be a solution to the Singular Perturbation problem
(\ref{eq:J07a}) and set $\phi^{\eps} = p^{\eps} + \eps q^{\eps}$ with
$p^{\eps} \in \GG$ and $q^{\eps} \in \AA$. This decomposition is
unique and we observe
\begin{gather}
  \left\{
    \begin{array}{ll}
      a_\perp (p^{\eps},v) + \eps a_\perp (q^{\eps},v) + a_\parallel
      (q^{\eps},v) = (f,v)  & \forall v\in \VV\,,\\[3mm]
      a_\parallel (p^{\eps},w) = 0 & \forall w\in \AA\,,
    \end{array}
  \right.
  \label{eq:Jo0a}
\end{gather}
or equivalently
\begin{gather}
  (AP)\,\,\,
  \left\{
    \begin{array}{ll}
      a (\phi^{\eps},v) + (1-\eps) a_\parallel (q^{\eps},v) = (f,v)  & \forall v\in \VV\,,\\[3mm]
      a_\parallel (\phi^{\eps},w) = \eps a_\parallel (q^{\eps},w) & \forall w\in \AA\,,
    \end{array}
  \right.
  \label{eq:AP2}
\end{gather}
with the bilinear form $a(v,w)$ defined as
\begin{gather}
  a(v,w) = \int_\Omega \mathbb A\nabla v \cdot \nabla w
  \label{eq:Jbab}.
\end{gather}
The matrix $\mathbb A$ is given by
\begin{gather}
  \mathbb A =  A_\parallel b \otimes b + (Id - b \otimes b)A_\perp (Id - b \otimes b)
  \label{eq:Jcab},
\end{gather}
and is $\eps$ independent, $\mathbb A = \mathbb A_1$. 

The above formulation is the Asymptotic Preserving
reformulation based on the Micro Macro decomposition.

%%%%%%%%%%%%%%%%%%%%%%%
\subsection{The stabilized Asymptotic Preserving approach (AP-model)}
%%%%%%%%%%%%%%%%%%%%%%%

The Asymptotic Preserving approach presented above has some
limitations originating in the choice of the vector space $\AA$. Note
that in the previous paper the uniqueness of $q^\varepsilon $ was ensured
by setting $q^\varepsilon $ to $0$ on the $\Gamma _{in}$ boundary
under hypothesis that every field line of $b$ has its beginning on
$\Gamma _{in}$ and an end on $\Gamma _{out}$. In other words, more
complex geometries, like for example closed field lines are not
permitted. In this paper we propose a new way of providing the uniqueness
of $q^\varepsilon $ which overcomes the limitations of our previous
method. The idea is based on the penalty stabilization method
introduced in \cite{BrezziDouglas88} for the Stokes problem.

Let us propose a new Asymptotic Preserving method:
find $(\phi ^\varepsilon ,q^\varepsilon )\in \VV \times \VV$ such that
\begin{gather}
  (APS)\,\,\,
  \left\{
    \begin{array}{ll}
      a (\phi^{\eps},v) + (1-\eps) a_\parallel (q^{\eps},v) = (f,v)  & \forall v\in \VV\,,\\[3mm]
      a_\parallel (\phi^{\eps},w) = \eps a_\parallel (q^{\eps},w) + 
      \sum_{K\in\tau_h } h_K^2 \int_K \mathbb A \nabla q^\varepsilon \cdot \nabla w & \forall w\in \VV\,,
      %%h^2 a(q^\varepsilon ,w) & \forall w\in \VV\,,
    \end{array}
  \right.
  \label{eq:AP2S}
\end{gather}
where $h_K$ denotes the size of the element $K$. Note that now, instead
of seeking $q^\varepsilon \in \AA$ we are looking for $q^\varepsilon
\in \VV$. Existence and uniqueness of the above problem can be easily
proved by the Lax-Milgram theorem.

%% The uniqueness of $q^\varepsilon $ is now provided by the
%% stabilization term.

%% Indeed, putting $f=0$ in (\ref{eq:AP2S}) and choosing $v=\phi ^\eps$,
%% $w=q^\eps$ immediately gives
%% \begin{gather}
  %% a(\phi ^\eps , \phi ^\eps) + (1-\eps)\eps a_\parallel (q^\eps ,
  %% q^\eps) + 
  %% \sum_{K\in\tau_h } h_K^2 \int_K \mathbb A \nabla q^\varepsilon \cdot
  %% \nabla q^\eps
  %% =0
  %% \label{eq:Jjdb}
%% \end{gather} 
%% and hence $\phi ^\eps,q^\eps = 0$, so the solution, if it exists, is unique.

%%%%%%%%%%%%%%%%%%%%%%%
\section{Numerical method}\label{sec:num_met}
%%%%%%%%%%%%%%%%%%%%%%%
This section concerns the discretization of the Asymptotic Preserving
formulation (\ref{eq:AP2S}), based on a finite element method. The
anisotropic error indicator is introduced and the obtained numerical
results are studied.

Let us denote by $\mathcal{V}_{h} \subset \mathcal{V}$ and $\mathcal{A}_{h}
\subset \mathcal{A}$ the finite dimensional approximation spaces,
constructed by means of $P_1$ finite elements. We are thus looking for
a discrete solution $(\phi^\varepsilon_h,\;q^\varepsilon_h) \in
\VV_h\times \AA_h$ of the following system
\begin{gather}
  (APS)_h\,\,\,
  \left\{
    \begin{array}{ll}
      a (\phi^{\eps}_h,v_h) + (1-\eps) a_\parallel (q^{\eps}_h,v_h) =
      (f_h,v_h)  & \forall v_h\in \VV_h\,,\\[3mm] 
      a_\parallel (\phi^{\eps}_h,w_h) = \eps a_\parallel (q_h^{\eps},w_h) + 
      \sum_{K\in\tau_h } h_K^2 \int_K \mathbb A \nabla q^\varepsilon_h \cdot \nabla w_h  & \forall w_h\in \VV_h\,.
    \end{array}
  \right.
  \label{eq:AP2disc}
\end{gather}

%%%%%%%%%%%%%%%%%%%%%%%
\subsection{Adaptive finite elements with large aspect ratio}
%%%%%%%%%%%%%%%%%%%%%%%

We now propose an adaptive finite element algorithm. The goal is to
build successive triangulations with large aspect ratio such that the
relative estimated error of the function $\phi^{\eps} =p^{\eps}+\eps
q^{\eps}$ in the $H^1(\Omega)$ norm is close to a preset tolerance
$TOL$. For this purpose, we introduce an error indicator which
requires some further notations. This error indicator measures the
error of the numerical solution $\phi^{\eps}$ in the directions of
maximum and minimum stretching of the triangle. The goal of the
adaptive algorithm is then to equidistribute the error indicator in
the directions of maximum and minimum stretching, and to align the
directions of maximum and minimum stretching with the directions of
maximum and minimum error. We refer to
\cite{picasso-03a,picasso-03,burman-picasso-03,FormaggiaPerotto01,FormaggiaPerotto03}
for theoretical justifications.

\begin{figure}[htbp]

  \vspace{5ex}

  \begin{center}
    \includegraphics[height=2cm]{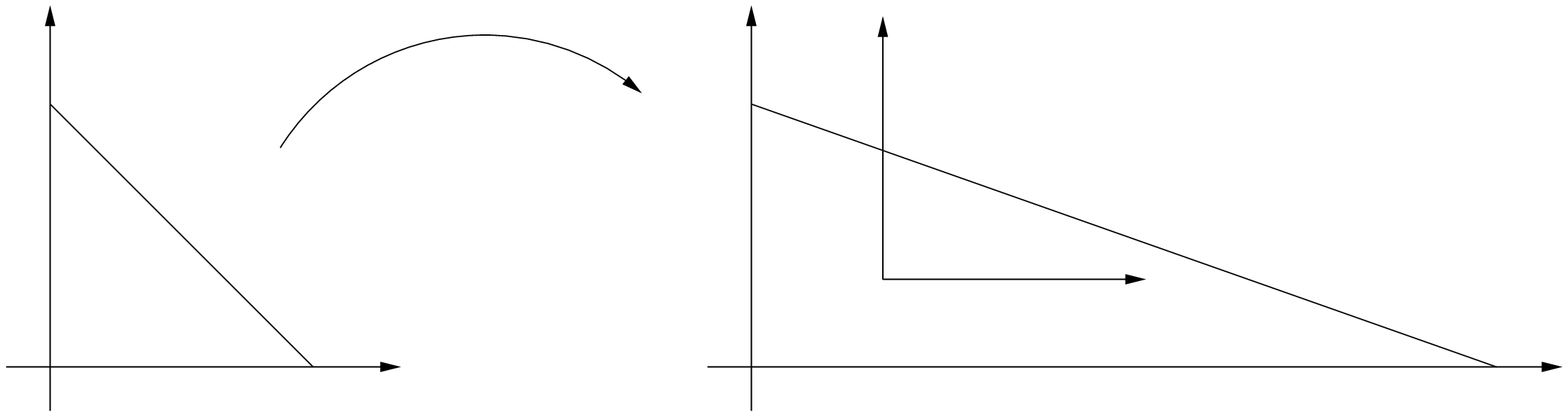}
  \end{center}
  \begin{picture}(0,0)(-75,5)
    
    \put(60,15){$\hat x_1$}
    \put(0,80){$\hat x_2$}
    \put(55,75){$T_K$}
    \put(225,20){$x_1$}
    \put(100,80){$x_2$}
    \put(40,10){$1$}
    \put(0,55){$1$}
    \put(200,10){$H$}
    \put(97,55){$h$}
    \put(160,35){$\mathbf r_{1,K}$}
    \put(125,70){$\mathbf r_{2,K}$}

  \end{picture}
  \caption{A simple example of transformation from reference triangle
    $\hat K$ to generic triangle $K$.}
  \label{fig_aniso}
\end{figure}

For any triangle $K$ of the mesh, let $T_K: \hat K\to K$ be the affine
transformation which maps the reference triangle $\hat K$ into
$K$. Let $M_K$ be the Jacobian of $T_K$ that is
\begin{gather}
  \mathbf x=T_K(\hat {\mathbf x})=M_K \hat {\mathbf x} + \mathbf t_K.
  \nonumber
\end{gather}
Since $M_K$ is invertible, it admits a singular value decomposition
$M_K=R_K^T\Lambda_K P_K$, where $R_K$ and $P_K$ are orthogonal and
where $\Lambda_K$ is diagonal with positive entries. In the following
we set
\begin{gather}
  \Lambda_K=\begin{pmatrix}\lambda_{1,K} & 0\nonumber\\ 0 & \lambda_{2,K}
  \end{pmatrix}
  \qquad\text{and}\qquad
  R_K=\begin{pmatrix}\mathbf r_{1,K}^T \nonumber\\ \mathbf r_{2,K}^T\end{pmatrix},
  \label{eq:lr}
\end{gather}
with the choice $\lambda_{1,K}\ge\lambda_{2,K}$. A simple example of
such a transformation is $x_1=H \hat x_1$, $x_2=h \hat x_2$, with
$H\ge h$, thus
\begin{gather}
  M_K=
  \begin{pmatrix} 
    H & 0 \nonumber\\ 
    0 & h 
  \end{pmatrix}
  \quad \lambda_{1,K}=H,
  \quad \lambda_{2,K}=h,
  \quad \mathbf r_{1,K}=\begin{pmatrix}1\nonumber\\0\end{pmatrix},
  \quad \mathbf r_{2,K}=\begin{pmatrix}0\nonumber\\1\end{pmatrix},
  \nonumber
\end{gather}
see Figure \ref{fig_aniso}. In other words $\mathbf r_{1,K}$ and
$\mathbf r_{2,K}$ are the directions of maximum and minimum
stretching, while $\lambda_{1,K}$ and $\lambda_{2,K}$ measure the
amplitude of stretching.

Let $I_h : H^{1}_0 (\Omega) \rightarrow \VV_h$ be a Cl\'ement or
Scott-Zhang like interpolation operator. We now recall some
interpolation results due to
\cite{FormaggiaPerotto01,FormaggiaPerotto03,MichelettiPerottoPicasso03}.

\begin{proposition}\label{prop:est}
  There is a constant $C = C (\hat{K})$ such that for all $v \in H^{1}
  (\Omega)$, for all $K \in \tau_h$, for all edges $e$ of $K$, we have
  \begin{gather}
    ||v-I_hv||_{L^{2} (\Omega)} \leq
    C
    \left(
      \lambda_{1,K}^{2}(\mathbf r_{1,K} G_K (v)\mathbf r_{1,K}) +
      \lambda_{2,K}^{2}(\mathbf r_{2,K} G_K (v)\mathbf r_{2,K})
    \right)^{1/2}
    \label{eq:Jq0a1}, \\
    ||v-I_hv||_{L^{2} (e)} \leq
    Ch_K^{1/2}
    \left(
      \frac{\lambda_{1,K}}{\lambda_{2,K}}(\mathbf r_{1,K} G_K (v)\mathbf r_{1,K}) +
      \frac{\lambda_{2,K}}{\lambda_{1,K}}(\mathbf r_{2,K} G_K (v)\mathbf r_{2,K})
    \right)^{1/2}
    \label{eq:Jq0a2}, \\
    ||\nabla(v-I_hv)||_{L^{2} (K)} \leq
    C
    \left(
      \frac{\lambda_{1,K}^{2}}{\lambda_{2,K}^{2}}(\mathbf r_{1,K} G_K (v)\mathbf r_{1,K}) +
      (\mathbf r_{2,K} G_K (v)\mathbf r_{2,K})
    \right)^{1/2}
    \label{eq:Jq0a3}.
  \end{gather}
  Here $h_k = \text{diam }K$, $\lambda_{i,K}$ and $\mathbf r_{i,K}$
  are given by (\ref{eq:lr}), and $G_K (v)$ denotes the $2\times 2$
  matrix defined as
  \begin{gather}
    G_K(v) =
    \begin{pmatrix}
      \displaystyle{\int_K 
        \left(\frac{\partial v}{\partial x_1}\right)^2 dx}
      & 
      \displaystyle{\int_K 
        \left(\frac{\partial v}{\partial x_1}\right)
        \left(\frac{\partial v}{\partial x_2}\right) dx}
      \\ 
      \displaystyle{\int_K 
        \left(\frac{\partial v}{\partial x_1}\right)
        \left(\frac{\partial v}{\partial x_2}\right) dx}
      & \displaystyle{\int_K 
        \left(\frac{\partial v}{\partial x_2}\right)^2 dx}
    \end{pmatrix}.
  \end{gather}
\end{proposition}
\begin{proof}
  The first estimate is in Proposition 3.1 of \cite{FormaggiaPerotto01}, the second
  estimate is in Proposition 2.2 of \cite{FormaggiaPerotto03}, the third estimate is in
  Proposition 2.5 of \cite{MichelettiPerottoPicasso03}. 
\end{proof}

The results of Proposition \ref{prop:est} are now used to derive an
anisotropic error indicator for the Asymptotic Preserving
reformulation. The error is first related to the equation
residual. The Cl\'ement interpolant is introduced. Then the
anisotropic interpolation results are used. Finally, a Zienkiewicz-Zhu
error estimator is used to approach the error gradient.

Let $e = \phi^{\eps} - \phi^{\eps}_h$ and $e_q = q^{\eps} -
q^{\eps}_h$. The following error estimate for the Asymptotic
Preserving reformulation (\ref{eq:AP2}) holds.

\begin{proposition}\label{prop:ee}
  There exist a constant $C$ depending only on the interpolation
  constants from Proposition \ref{prop:est} and not on the mesh size
  nor aspect ratio such that
  \begin{multline}
    \int_\Omega \mathbb A \nabla e \cdot \nabla e
    + (1-\varepsilon ) \varepsilon 
    \int_\Omega A_\parallel \nabla_\parallel e_q \cdot \nabla_\parallel e_q
    + (1-\eps)
    \sum_{K\in\tau_h } h_K^2 \int_K \mathbb A \nabla e_q \cdot\nabla e_q 
    %% \int_\Omega A \nabla e \cdot \nabla e
    \leq 
    \\
    \
    C \sum_{K \in \tau_h}
    \Bigg(
      ||f+\nabla \cdot (\mathbb A \nabla \phi^{\eps}_h) + 
      (1-\eps) \nabla_\parallel \cdot (A_\parallel \nabla_\parallel
      q^{\eps}_h )||_{L^{2}(K)}
      \\
      + \frac{1}{2\lambda_{2,K}^{1/2}} ||[\mathbb A\nabla \phi^{\eps}_h \cdot
      n] ||_{L^{2}(\partial K)}
      + \frac{1-\varepsilon }{2\lambda_{2,K}^{1/2}} ||[A_\parallel\nabla_\parallel
      q^{\eps}_h \cdot n] ||_{L^{2}(\partial K)}
    \Bigg) \\
    \times
    \left(
    \lambda_{1,K}^{2}(\mathbf r_{1,K} G_K (e)\mathbf r_{1,K}) +
    \lambda_{2,K}^{2}(\mathbf r_{2,K} G_K (e)\mathbf r_{2,K})
    \right)^{1/2}
    \\
    +
    (1-\eps )
    \Bigg(
    ||\nabla_\parallel \cdot (A_\parallel \nabla_\parallel
    (\phi^\eps_h-\eps q^\eps_h) )||_{L^{2}(K)}
    +
    \frac{1}{2\lambda_{2,K}^{1/2}} ||[A_\parallel\nabla_\parallel
      q^{\eps}_h \cdot n] ||_{L^{2}(\partial K)}
    \\
    +\lambda_{2,K}^2 || \nabla \cdot (\mathbb A \nabla q^\eps_h ) ||_{L^{2}(K)}
    + \lambda_{2,K}^{3/2}
    || \mathbb A \nabla q^\eps_h  \cdot n||_{L^{2}(\partial K)}    
    \Bigg)
    \\
    \times
    \left(
    \lambda_{1,K}^{2}(\mathbf r_{1,K} G_K (e_q)\mathbf r_{1,K}) +
    \lambda_{2,K}^{2}(\mathbf r_{2,K} G_K (e_q)\mathbf r_{2,K})
    \right)^{1/2}
    \label{eq:Js0a}.
  \end{multline}
  Here $[\cdot]$ denotes the jump of the bracketed quantity across an
  internal edge, $[\cdot]=0$ for an edge on the boundary
  $\partial\Omega_D$, $[\cdot]$ is set to twice the imposed flux on
  the $\partial\Omega_{in} \cup \partial\Omega_{out}$ and $n$ is the
  unit edge normal in arbitrary direction.
\end{proposition}
\begin{proof}
  Setting $v = e$ in the AP reformulation (\ref{eq:AP2}) yields
  \begin{gather}
    a(e,e) + (1-\varepsilon ) a_\parallel(e, e_q) =
    (f,e) - a(\phi^{\eps}_h , e) - (1-\varepsilon ) a_\parallel (q^\eps_h , e)
    \label{eq:J40a}.
  \end{gather}
  Now, since $a_\parallel (\phi^\eps-\eps q^\eps ,e_q) =
  \sum_{K\in\tau_h } h_K^2 \int_K \mathbb A \nabla q^\eps \cdot \nabla e_q $ 
  we obtain
  \begin{gather}
    a_\parallel(e, e_q) = \eps a_\parallel(e_q,e_q) +
    \sum_{K\in\tau_h } h_K^2 \int_K \mathbb A \nabla q^\eps \cdot
    \nabla e_q 
    - a_\parallel(\phi ^\eps_h -\eps q^\eps _h, e_q)
    \label{eq:Jkdb}
  \end{gather}
  and hence
  \begin{multline}
    \int_\Omega \mathbb A \nabla e \cdot \nabla e 
    + (1-\varepsilon ) \varepsilon \int_\Omega A_\parallel
    \nabla_\parallel e_q \cdot \nabla_\parallel e_q 
    +\sum_{K\in\tau_h } h_K^2 \int_K \mathbb A \nabla e_q \cdot\nabla e_q 
    =
    \\
    \int_\Omega fe - \int_\Omega \mathbb A \nabla\phi^{\eps}_h \cdot \nabla e
    \ - (1-\varepsilon )\! \int_\Omega A_\parallel  \nabla_\parallel q_h
    \cdot \nabla_\parallel e
    \ + (1-\varepsilon )\! \int_\Omega A_\parallel  \nabla_\parallel (\phi^{\eps}_h-\eps q^\eps_h)
    \cdot \nabla_\parallel e_q
    \\ - (1-\varepsilon) \sum_{K\in\tau_h } h_K^2 \int_K \mathbb A \nabla q^\eps_h \cdot\nabla e_q 
    \label{eq:J50a}.
  \end{multline}
  For any $v\in V$ we have 
  \begin{multline}
    (f,v) - a(\phi^{\eps}_h, v)
    - (1-\eps ) a_\parallel (q^\eps_h ,  v)
    \\
    \qquad =(f,v-I_hv) - a(\phi^{\eps}_h,v-I_hv)
    - (1-\eps ) a_\parallel (q^\eps_h ,v-I_hv)
    \hfill
    \\
    \qquad=\sum_{K \in \tau_h}
    \Bigg(
    \int_K (f+\nabla \cdot (\mathbb A \nabla \phi^{\eps}_h) + 
      (1-\eps) \nabla_\parallel \cdot (A_\parallel \nabla_\parallel
      q^{\eps}_h )) (v - I_h v) \hfill
      \\
      \qquad\qquad
      + \frac{1}{2} \int_{\partial K} [\mathbb A\nabla \phi^{\eps}_h \cdot n](v - I_h v)
      + \frac{1-\eps}{2} \int_{\partial K} [A_\parallel\nabla_\parallel q^{\eps}_h \cdot n](v - I_h v)
    \Bigg)
    \hfill
    \label{eq:J70a}.
  \end{multline}
  Furthermore, for any $w\in A$ the following holds true :
  \begin{multline}
    a_\parallel (\phi^\eps_h -\eps q^\eps_h, w)
    - \sum_{K\in\tau_h } h_K^2 \int_K \mathbb A \nabla q^\eps_h \cdot\nabla w 
    \\
    \qquad
    = 
    a_\parallel (\phi^\eps_h -\eps q^\eps_h, w - I_h w) 
    - \sum_{K\in\tau_h } h_K^2 \int_K \mathbb A \nabla q^\eps_h \cdot\nabla (w - I_h w) 
    \\
    \qquad
    =
    \sum_{K \in \tau_h}
    \Bigg(
    \int_K \nabla_\parallel \cdot (A_\parallel \nabla_\parallel (\phi^{\eps}_h-\eps q^\eps_h)) 
    (w - I_h v) \hfill
    + \frac{1}{2} \int_{\partial K} [A_\parallel\nabla_\parallel
    (\phi^{\eps}_h-\eps q^\eps_h) \cdot n](w - I_h w) \hfill
    \\- h_K^2 \int_K \nabla \cdot (\mathbb A \nabla q^\eps_h ) (w - I_h    w) 
    + h_K^2 \int_{\partial K} (\mathbb A \nabla q^\eps_h \cdot n) (w - I_h w) 
    \Bigg)
    \label{eq:J80a} .
  \end{multline}
  Now, choosing $v =e$, $w = e_q$ and using the Cauchy-Schwartz
  inequality together with the interpolation results of the
  Proposition \ref{prop:est} the following is obtained:
  \begin{multline}
    \int_\Omega \mathbb A \nabla e \cdot \nabla e
    + (1-\varepsilon ) \varepsilon 
    \int_\Omega A_\parallel \nabla_\parallel e_q \cdot \nabla_\parallel e_q
    \\ +
    (1-\varepsilon ) \sum_{K\in\tau_h } h_K^2 \int_K \mathbb A \nabla e_q \cdot\nabla e_q 
    \leq 
    C \sum_{K \in \tau_h}
    \Bigg(
      ||f+\nabla \cdot (\mathbb A \nabla \phi^{\eps}_h) + 
      (1-\eps) \nabla_\parallel \cdot (A_\parallel \nabla_\parallel
      q^{\eps}_h )||_{L^{2}(K)}
      \\
      + \frac{1}{2}\left(\frac{h_K}{\lambda_{1,K}\lambda_{2,K}}\right)^{1/2} 
      ||[\mathbb A\nabla \phi^{\eps}_h \cdot n] ||_{L^{2}(\partial K)}
      + \frac{1-\varepsilon }{2}\left(\frac{h_K}{\lambda_{1,K}\lambda_{2,K}}\right)^{1/2} 
      ||[A_\parallel\nabla_\parallel q^{\eps}_h \cdot n] ||_{L^{2}(\partial K)}
    \Bigg) \\
    \times
    \left(
      \lambda_{1,K}^{2}(\mathbf r_{1,K} G_K (e)\mathbf r_{1,K}) +
      \lambda_{2,K}^{2}(\mathbf r_{2,K} G_K (e)\mathbf r_{2,K})
    \right)^{1/2}
    \\
    +
    (1-\varepsilon )
    \Bigg(
      ||\nabla_\parallel \cdot (A_\parallel \nabla_\parallel (\phi^{\eps}_h-\eps q^\eps_h))||_{L^{2}(K)}
      + \frac{1}{2}\left(\frac{h_K}{\lambda_{1,K}\lambda_{2,K}}\right)^{1/2} \!\!
      ||[A_\parallel\nabla_\parallel (\phi^{\eps}_h-\eps q^\eps_h) \cdot n] ||_{L^{2}(\partial K)}
      \\
      + h_K^2       || \nabla \cdot (\mathbb A \nabla q^\eps_h ) ||_{L^{2}(K)}
      + \left(\frac{h^5_K}{\lambda_{1,K}\lambda_{2,K}}\right)^{1/2} 
      || \mathbb A \nabla q^\eps_h  \cdot n||_{L^{2}(\partial K)}
    \Bigg) 
    \\
    \times
    \left(
      \lambda_{1,K}^{2}(\mathbf r_{1,K} G_{K} (e_q)\mathbf r_{1,K}) +
      \lambda_{2,K}^{2}(\mathbf r_{2,K} G_{K} (e_q)\mathbf r_{2,K})
    \right)^{1/2}
    \label{eq:J90a}
  \end{multline}
  where $C = C (\hat{K})$.
  Since $\int_\Omega A_\parallel \nabla_\parallel e_q \cdot
  \nabla_\parallel e_q \geq 0$ and 
  \begin{gather}
    \lambda_{1,K} h_{\hat{K}} \leq h_K \leq \lambda_{2,K} h_{\hat{K}} 
    \label{eq:J00a},
  \end{gather}
  the inequality (\ref{eq:Js0a}) holds true.
\end{proof}

\begin{remark}
  Note that the above result does not contain any terms inversely
  proportional to $\varepsilon $ as it involves matrix $\mathbb A$ rather than
  $\mathbb A_\varepsilon$. The standard anisotropic error indicator for an
  anisotropic diffusion problem studied in \cite{picasso-03,picasso-06} takes form:
  \begin{multline}
    \int_\Omega \mathbb A_\varepsilon \nabla e \cdot \nabla e
    \leq 
    C \sum_{K \in \tau_h}
    \Bigg(
      ||f+\nabla \cdot (\mathbb A_\varepsilon \nabla \phi^{\eps}_h)
      + \frac{1}{2\lambda_{2,K}^{1/2}} ||[\mathbb A_\varepsilon \nabla \phi^{\eps}_h \cdot
      n] ||_{L^{2}(\partial K)}
    \Bigg) 
    \\
    \times
    \left(
      \lambda_{1,K}^{2}(\mathbf r_{1,K} G_K (e)\mathbf r_{1,K}) +
      \lambda_{2,K}^{2}(\mathbf r_{2,K} G_K (e)\mathbf r_{2,K})
    \right)^{1/2} ,
    \label{eq:Jaab}
  \end{multline}
  thus it involves terms of the order $\frac{1}{\varepsilon}$. While
  this error indicator remains valid it is of no practical use for
  small values of $\varepsilon $. Indeed, the remeshing algorithm
  which aims in keeping the error indicator close to a given value
  would yield meshes with mesh size proportional to $\varepsilon $.
\end{remark}
\begin{remark}
  In the case of $\varepsilon = 1$ the above error indicator reduces
  to the standard anisotropic error indicator for a diffusion problem
  studied in \cite{} :
  \begin{multline}
    \int_\Omega \mathbb A \nabla e \cdot \nabla e
    \leq 
    C \sum_{K \in \tau_h}
    \Bigg(
      ||f+\nabla \cdot (\mathbb A \nabla \phi^{\eps}_h)
      + \frac{1}{2\lambda_{2,K}^{1/2}} ||[\mathbb A\nabla \phi^{\eps}_h \cdot
      n] ||_{L^{2}(\partial K)}
    \Bigg)
    \\
    \times
    \left(
      \lambda_{1,K}^{2}(\mathbf r_{1,K} G_K (e)\mathbf r_{1,K}) +
      \lambda_{2,K}^{2}(\mathbf r_{2,K} G_K (e)\mathbf r_{2,K})
    \right)^{1/2}
    \label{eq:Jaab2}.
  \end{multline}
\end{remark}

Estimate (\ref{eq:Js0a}) is not a usual a posteriori error estimate as
it involves $\phi^{\eps}$ and $q^{\eps}$ on the right hand side. If we
can guess $\phi^{\eps} - \phi^{\eps}_h$ and $q^{\eps} - q^{\eps}_h$,
(\ref{eq:Js0a}) can be used to derive an anisotropic error
indicator. In order to do that, we introduce an error estimator based
on the superconvergent gradient recovery, namely Zienkiewicz Zhu like
error estimator \cite{AZCZ89,ZZ87,ZZ92} in its simplest form as defined in \cite{AinsworthOden97,Rodriguez94},
{\it i.e.} the difference between $\nabla \phi^{\eps}_h$ resp. $\nabla
q^{\eps}_h$ and an approximate $L^{2}$ projection of $\nabla
\phi^{\eps}_h$ resp. $\nabla q^{\eps}_h$ onto $\VV^{2}$ :
\begin{gather}
  \mathbf \eta^{ZZ} (\phi^{\eps}_h) =
  \begin{pmatrix}
	\eta_{1}^{ZZ}(\phi^{\eps}_h)
	\\
	\eta_{2}^{ZZ}(\phi^{\eps}_h)
  \end{pmatrix}
  =
  \begin{pmatrix}
	(I-\Pi_h)\left(\pd{\phi^{\eps}_h}{x_1}\right)
	\\
	(I-\Pi_h)\left(\pd{\phi^{\eps}_h}{x_2}\right)
  \end{pmatrix},
  \label{eq:Jt0a}
\end{gather}
where $\Pi_h$ is the projection operator which builds values at
vertices $P$ from constant values on triangles using the formula
\begin{gather}
  \begin{pmatrix}
    \Pi_h\left(\pd{\phi^{\eps}_h}{x_1}\right)(P)
    \nonumber\\
    \Pi_h\left(\pd{\phi^{\eps}_h}{x_2}\right)(P)
  \end{pmatrix}
  =
  \dfrac{1}
  {\displaystyle{\sum_{\underset{P\in K}{\text{tria. }K}}|K|}}
  \begin{pmatrix}
    {\displaystyle{\sum_{\underset{P\in K}{\text{tria. }K}}|K| 
        \left(\pd{\phi^{\eps}_h}{x_1}\right)_{|K}}}
    \nonumber\\
    {\displaystyle{\sum_{\underset{P\in K}{\text{tria. }K}}|K| 
        \left(\pd{\phi^{\eps}_h}{x_2}\right)_{|K}}}
  \end{pmatrix}.
  \nonumber
\end{gather}
Z-Z like error estimator is asymptotically exact for a parallel meshes
and smooth solutions \cite{AinsworthOden97,Rodriguez94}. Our error indicator is obtained by
replacing the matrices $G_K (e)$ and $G_K (e_q)$ by approximate ones
$\tilde{G}_K (\phi^{\eps}_h)$ and $\tilde{G}_K (q^{\eps}_h)$ defined by
\begin{gather}
  \tilde{G}_K (\phi^{\eps}_h)
  =
  \begin{pmatrix}
    \displaystyle{\int_K (\eta_{1}^{ZZ}(\phi^{\eps}_h))^2 dx}
    & \displaystyle{\int_K \eta_{1}^{ZZ}(\phi^{\eps}_h)\eta_{2}^{ZZ}(\phi^{\eps}_h)dx}
    \\ 
    \displaystyle{\int_K \eta_{1}^{ZZ}(\phi^{\eps}_h)\eta_{2}^{ZZ}(\phi^{\eps}_h) dx}
    & \displaystyle{\int_K (\eta_{2}^{ZZ}(\phi^{\eps}_h))^2 dx}
  \end{pmatrix}.
  \label{eq:Ju0a}
\end{gather}
The anisotropic error indicator defined on each triangle $K$ takes the form
\begin{multline}
  \Big(\eta_K^A (\phi^{\eps}_h, q^{\eps}_h))\Big)^{2} =
    \Bigg(
      ||f+\nabla \cdot (\mathbb A \nabla \phi^{\eps}_h) + 
      (1-\eps) \nabla_\parallel \cdot (A_\parallel \nabla_\parallel
      q^{\eps}_h )||_{L^{2}(K)}
      \\
      + \frac{1}{2\lambda_{2,K}^{1/2}} ||[\mathbb A\nabla \phi^{\eps}_h \cdot
      n] ||_{L^{2}(\partial K)}
      + \frac{1-\varepsilon }{2\lambda_{2,K}^{1/2}} ||[A_\parallel\nabla_\parallel
      q^{\eps}_h \cdot n] ||_{L^{2}(\partial K)}
    \Bigg) \\
    \times
    \left(
      \lambda_{1,K}^{2}(\mathbf r_{1,K} \tilde{G}_K (\phi^{\eps}_h)\mathbf r_{1,K}) +
      \lambda_{2,K}^{2}(\mathbf r_{2,K} \tilde{G}_K (\phi^{\eps}_h)\mathbf r_{2,K})
    \right)^{1/2}
    \\
    +
    (1-\varepsilon )
    \Bigg(
      ||\nabla_\parallel \cdot (A_\parallel \nabla_\parallel (\phi^\eps_h-\eps q^\eps_h))||_{L^{2}(K)}
      + \frac{1}{2\lambda_{2,K}^{1/2}} ||[A_\parallel\nabla_\parallel
      (\phi^\eps_h-\eps q^\eps_h) \cdot n] ||_{L^{2}(\partial K)}
      \\
      +\lambda_{2,K}^2 || \nabla \cdot (\mathbb A \nabla q^\eps_h ) ||_{L^{2}(K)}
      + \lambda_{2,K}^{3/2}
      || \mathbb A \nabla q^\eps_h  \cdot n||_{L^{2}(\partial K)}    
    \Bigg) 
    \\
    \times
    \left(
      \lambda_{1,K}^{2}(\mathbf r_{1,K} \tilde{G}_{K} (q^{\eps}_h)\mathbf r_{1,K}) +
      \lambda_{2,K}^{2}(\mathbf r_{2,K} \tilde{G}_{K} (q^{\eps}_h)\mathbf r_{2,K})
    \right)^{1/2}
  \label{eq:Jdab}.
\end{multline}
Introducing
\begin{align}
  \rho_{\phi,K}
  &=
  ||f+\nabla \cdot (\mathbb A \nabla \phi^{\eps}_h) + 
  (1-\eps) \nabla_\parallel \cdot (A_\parallel \nabla_\parallel
  q^{\eps}_h )||_{L^{2}(K)}
  \\
  & \quad+ \frac{1}{2\lambda_{2,K}^{1/2}} ||[\mathbb A\nabla \phi^{\eps}_h \cdot
  n] ||_{L^{2}(\partial K)}
  + \frac{1-\varepsilon }{2\lambda_{2,K}^{1/2}} ||[A_\parallel\nabla_\parallel
  q^{\eps}_h \cdot n] ||_{L^{2}(\partial K)} ,
  \\
  \Big(\eta_{\phi ,K}^A (\phi^{\eps}_h, q^{\eps}_h))\Big)^{2}
  &=
  \rho_{\phi,K}
  \left(
  \lambda_{1,K}^{2}(\mathbf r_{1,K} \tilde{G}_K (\phi^{\eps}_h)\mathbf r_{1,K}) +
  \lambda_{2,K}^{2}(\mathbf r_{2,K} \tilde{G}_K (\phi^{\eps}_h)\mathbf r_{2,K})
  \right)^{1/2} 
  \label{eq:Jmdb}
\end{align}
and
\begin{align}
  \rho_{q,K}
  &=
    (1-\varepsilon )
    \Bigg(
      ||\nabla_\parallel \cdot (A_\parallel \nabla_\parallel (\phi^\eps_h-\eps q^\eps_h))||_{L^{2}(K)}
      \\
      &\quad+ \frac{1}{2\lambda_{2,K}^{1/2}} ||[A_\parallel\nabla_\parallel
      (\phi^\eps_h-\eps q^\eps_h) \cdot n] ||_{L^{2}(\partial K)}
      \\
      &\quad+\lambda_{2,K}^2 || \nabla \cdot (\mathbb A \nabla q^\eps_h ) ||_{L^{2}(K)}
      + \lambda_{1,K}^{3/2}
      || \mathbb A \nabla q^\eps_h  \cdot n||_{L^{2}(\partial K)}    
      \Bigg),
      \\
      \Big(\eta_{q,K}^A (\phi^{\eps}_h, q^{\eps}_h))\Big)^{2}
      &=
      \rho_{a,K}
      \left(
      \lambda_{1,K}^{2}(\mathbf r_{1,K} \tilde{G}_K (\phi^{\eps}_h)\mathbf r_{1,K}) +
      \lambda_{2,K}^{2}(\mathbf r_{2,K} \tilde{G}_K (\phi^{\eps}_h)\mathbf r_{2,K})
      \right)^{1/2} 
      \label{eq:Jndb}
\end{align}
allows to introduce a more compact notation
\begin{gather}
  \Big(\eta_K^A (\phi^{\eps}_h, q^{\eps}_h))\Big)^{2} 
   =
  \Big(\eta_{\phi ,K}^A (\phi^{\eps}_h, q^{\eps}_h))\Big)^{2}
  + \Big(\eta_{q,K}^A (\phi^{\eps}_h, q^{\eps}_h))\Big)^{2}
  \label{eq:J7ab}.
\end{gather}

\subsection{Adaptive algorithm}

The goal of our adaptive algorithm is to build a triangulation such
that the error is equidistributed in the direction of the maximal and
minimal stretching of triangles and the relative global error
indicator is closed to prescribed tolerance $TOL$. We have
\begin{equation}
  0.75\ TOL\le
  \dfrac{
    \Bigl(\eta^A (\phi^{\eps}_h, q^{\eps}_h)\Bigr)
  }
  {
    \sqrt{\displaystyle{\int_\Omega|\nabla \phi^{\eps}_h|^2}}
  }
  \le 1.25\ TOL.
  \label{goal_adap}
\end{equation}
with
\begin{gather}
  \Bigl(\eta^A (\phi^{\eps}_h, q^{\eps}_h)\Bigr)^2
  =
  \displaystyle{\sum_{\text{tria. }K}
    \Bigl(\eta_K^A (\phi^{\eps}_h, q^{\eps}_h)\Bigr)^2}
  \label{eq:Jfab}.
\end{gather}

A sufficient condition to satisfy (\ref{goal_adap}) is to build a
triangulation with large aspect ratio such that
\begin{equation*}
  \begin{split}
	\dfrac{0.75^2 TOL^2}{NT}
	\int_\Omega|\nabla \phi^{\eps}_h|^2
	\le
	\Bigl(\eta^A_{K}(\phi^{\eps}_h, q^{\eps}_h)\Bigr)^2
	\le \dfrac{1.25^2 TOL^2}{NT}
	\int_\Omega|\nabla \phi^{\eps}_h|^2
  \end{split}
\end{equation*}
for each triangle $K$, where $NT$ is the number of triangles of the
mesh. Since the mesh generator \texttt{BL2D} mesh generator used in
our simulations \cite{bl2d} requires data on the mesh vertices rather
than on the triangles, we need to translate the above local triangle
condition into a condition for mesh for the mesh vertices. Let us
introduce a point defined error indicator :
\begin{gather}
  \eta_P^A (\phi^{\eps}_h, q^{\eps}_h) =
  \left(
    \sum_{\underset{P\in K}{\text{tria. }K}} 
    \Bigl(\eta^A_{K}(\phi^{\eps}_h, q^{\eps}_h)\Bigr)^4
  \right)^{1/4}
  \label{eq:Jgab}
\end{gather}
and hence
\begin{gather}
  \sum_P
  \Bigl(\eta^A_{P}(\phi^{\eps}_h, q^{\eps}_h)\Bigr)^4
  =
  3
  \sum_K
  \Bigl(\eta^A_{K}(\phi^{\eps}_h, q^{\eps}_h)\Bigr)^4
  \label{eq:Jhab}.
\end{gather}
Therefore, the following local condition holds
\begin{gather}
  \begin{split}
    \dfrac{\sqrt{3}}{NV}0.75^2 TOL^2
    \int_\Omega|\nabla \phi^{\eps}_h|^2
    \le
    \Bigl(\eta^A_{P}(\phi^{\eps}_h, q^{\eps}_h)\Bigr)^2
    \le \dfrac{\sqrt{3}}{NV}1.25^2 TOL^2
    \int_\Omega|\nabla \phi^{\eps}_h|^2
  \end{split}
  \label{eq:Jiab}
\end{gather}
where $NV$ is a number of mesh vertices. Then, we define
$\eta^{A}_{i,P} (\phi^{\eps}_h, q^{\eps}_h)$, with $i = 1,2$ at the
mesh nodes
\begin{gather}
  \Bigl(\eta^A_{i,P}(\phi^{\eps}_h, q^{\eps}_h)\Bigr)^4
  =
  \sum_{\underset{P\in K}{\text{tria. }K}} 
  \lambda_{i,K}^{2}
  \left(
    \mathbf r_{i,K} 
    \left(  
      \rho_{\phi,K}^2\tilde{G}_K (\phi^{\eps}_h)
      +
      \rho_{q,K}^2
      \tilde{G}_{K} (q^{\eps}_h)
    \right)
    \mathbf r_{i,K} 
  \right)
  \label{eq:Jjab}.
\end{gather}
The value of $\eta^A_{i,P}(\phi^{\eps}_h, q^{\eps}_h)$ represents the
error in the direction of the maximum and minimum stretching of the
triangle $K$. We note that the point error indicator is bounded by 
\begin{multline}
  \Bigl(\eta^A_{1,P}(\phi^{\eps}_h, q^{\eps}_h)\Bigr)^4 +
  \Bigl(\eta^A_{2,P}(\phi^{\eps}_h, q^{\eps}_h)\Bigr)^4
  \leq
  \Bigl(\eta^A_{P}(\phi^{\eps}_h, q^{\eps}_h)\Bigr)^4
  \\
  \qquad\leq
  2 \left( \Bigl(\eta^A_{1,P}(\phi^{\eps}_h, q^{\eps}_h)\Bigr)^4 +
  \Bigl(\eta^A_{2,P}(\phi^{\eps}_h, q^{\eps}_h)\Bigr)^4 \right)
  \label{eq:Jqdb}.
\end{multline}

The mesh adaptation algorithm can be summarized as follows. For all
vertices $P$ of the mesh $\eta^A_{1,P} (\phi^{\eps}_h, q^{\eps}_h)$
and $\eta^A_{2,P} (\phi^{\eps}_h, q^{\eps}_h)$ are
computed. Furthermore, we compute $\lambda_{1,P}$ and $\lambda_{2,P}$
as an average of the $\lambda_{1,K}$ and $\lambda_{2,K}$ of the
neighboring triangles $K$.
  
The input data for the \texttt{BL2D} mesh generator is computed: the
stretching amplitude $h_{i,P}$, $i=1,2$ and the direction of the
anisotropy $\theta_P$. In the first step new $h_{i,P}$ are
obtained. For every mesh point $P$, if
\begin{gather}
  4 \Bigl(\eta^A_{i,P}(\phi^{\eps}_h, q^{\eps}_h)\Bigr)^4 <
  \dfrac{3}{(NV)^2}0.75^4 TOL^4
  \left( \int_\Omega|\nabla \phi^{\eps}_h|^2\right)^2
  \label{eq:Jmab}
\end{gather}
then $h_{i,P}$ is set to $3/2 \lambda_{i,P}$. If
\begin{gather}
  2 \Bigl(\eta^A_{i,P}(\phi^{\eps}_h, q^{\eps}_h)\Bigr)^4 >
  \dfrac{3}{(NV)^2}1.25^4 TOL^4
  \left( \int_\Omega|\nabla \phi^{\eps}_h|^2\right)^2
  \label{eq:Jnab}
\end{gather}
then $h_{i,P}$ is set to $2/3 \lambda_{i,P}$. Otherwise, $h_{i,P}$ is
set to $\lambda_{i,P}$.

In the second step of the mesh adaptation the new anisotropy direction
is found. For every mesh point average matrices $\tilde{G}_P
(\phi^\eps_h)$ and $\tilde{G}_{P} (q^\eps_h)$ are calculated. The
angle $\theta_P$ is set to the angle between the eigenvector
corresponding to the largest eigenvalue of the matrix
\begin{gather}
  \rho_{\phi,K}^2\tilde{G}_K (\phi^{\eps}_h)
  +
  \rho_{q,K}^2
  \tilde{G}_{K} (q^{\eps}_h)
  \label{eq:J8ab}
\end{gather}
and the $Ox$ direction. Finally, new mesh is generated using the
\texttt{BL2D} mesh generator.

\subsection{Simplified error indicator}

The anisotropic error indicator introduced in the previous sections
involves the term $\tilde{G}_K (q^\eps_h)$. This means that the
perpendicular derivatives of $q^\eps_h$ will play role in the error
estimation procedure. This is not necessarily desirable since in some
cases this may result in mesh over-refinement in the direction
perpendicular to the anisotropy direction. That is to say the adaptive
algorithm could continue to refine the mesh in the perpendicular
direction without any increase of precision in $\phi^\eps_h$. This is
why we propose an alternative approach where the simplified error
indicator is related only to the residue of the first equation and the
matrix $\tilde{G}_K (\phi^\eps_h)$:
\begin{multline}
  \Big(\eta_K^{SA} (\phi^{\eps}_h, q^{\eps}_h)\Big)^{2} =
  \Bigg(
  ||f+\nabla \cdot (\mathbb A \nabla \phi^{\eps}_h) + 
  (1-\eps) \nabla_\parallel \cdot (A_\parallel \nabla_\parallel
  q^{\eps}_h )||_{L^{2}(K)}
  \\
  + \frac{1}{2\lambda_{2,K}^{1/2}} ||[\mathbb A\nabla \phi^{\eps}_h \cdot
    n] ||_{L^{2}(\partial K)}
  + \frac{1-\varepsilon }{2\lambda_{2,K}^{1/2}} ||[A_\parallel\nabla_\parallel
    q^{\eps}_h \cdot n] ||_{L^{2}(\partial K)}
  \Bigg) \\
  \times
  \left(
  \lambda_{1,K}^{2}(\mathbf r_{1,K} \tilde{G}_K (\phi^{\eps}_h)\mathbf r_{1,K}) +
  \lambda_{2,K}^{2}(\mathbf r_{2,K} \tilde{G}_K (\phi^{\eps}_h)\mathbf r_{2,K})
  \right)^{1/2}
  \label{eq:Jdab2},
\end{multline}
or in more compact notation:
\begin{gather}
  \Big(\eta_K^{SA} (\phi^{\eps}_h, q^{\eps}_h))\Big)^{2} =
  \Big(\eta_{\phi ,K}^A (\phi^{\eps}_h, q^{\eps}_h))\Big)^{2}
  \label{eq:Jebb}.
\end{gather}
As in the previous section the nodal simplified error indicator is
defined:
\begin{gather}
  \sum_P
  \Bigl(\eta^{SA}_{P}(\phi^{\eps}_h, q^{\eps}_h)\Bigr)^4
  =
  3
  \sum_K
  \Bigl(\eta^{SA}_{K}(\phi^{\eps}_h, q^{\eps}_h)\Bigr)^4
  ,
  \\
  \Bigl(\eta^{SA}_{i,P}(\phi^{\eps}_h, q^{\eps}_h)\Bigr)^4
  =
  \sum_{\underset{P\in K}{\text{tria. }K}} 
  \rho_{\phi,K}^2
  \lambda_{i,K}^{2}
  \mathbf r_{i,K} 
  \tilde{G}_{K} (\phi^{\eps}_h)
  \mathbf r_{i,K} 
  \label{eq:Jfbb}.
\end{gather}
The obtained adaptive algorithm is almost the same as before. Only now
$\eta^{A}_{i,P}$ is replaced by a simplified version $\eta^{A}_{i,P}$,
the coarsening criterion is slightly changed :
 if
\begin{gather}
  2 \Bigl(\eta^{SA}_{i,P}(\phi^{\eps}_h, q^{\eps}_h)\Bigr)^4 <
  \dfrac{3}{(NV)^2}0.75^4 TOL^4
  \left( \int_\Omega|\nabla \phi^{\eps}_h|^2\right)^2
  \label{eq:Jmab1}
\end{gather}
then $h_{i,P}$ is set to $3/2 \lambda_{i,P}$. If
\begin{gather}
  2 \Bigl(\eta^{SA}_{i,P}(\phi^{\eps}_h, q^{\eps}_h)\Bigr)^4 >
  \dfrac{3}{(NV)^2}1.25^4 TOL^4
  \left( \int_\Omega|\nabla \phi^{\eps}_h|^2\right)^2
  \label{eq:Jnab1}
\end{gather}
then $h_{i,P}$ is set to $2/3 \lambda_{i,P}$. Otherwise, $h_{i,P}$ is
set to $\lambda_{i,P}$.
Finally, the mesh anisotropy direction is aligned with the largest
eigenvalue of the matrix $\tilde{G}_K(\phi^\eps_h)$.

\subsection{Numerical results}\label{sec:test case}

\subsubsection{Numerical study of the effectivity index and the
  convergence of the stabilized AP scheme}

Let us define 
\begin{gather}
  \eta^{ZZ} =\left( \sum_{K\in \tau_h} \int_K |\mathbf \eta^{ZZ}(\phi^{\eps}_h)|^2 \right)^{1/2} ,
  \\
  \eta^{A}  =\left( \sum_{K\in \tau_h} \int_K (\mathbf \eta^{A} (\phi^{\eps}_h))^2 \right)^{1/2} ,\\
  \eta^{SA} =\left( \sum_{K\in \tau_h} \int_K (\mathbf \eta^{SA}(\phi^{\eps}_h))^2 \right)^{1/2} ,
  \label{eq:Jgbb}
\end{gather}
the Z-Z error estimator, the anisotropic error estimator and the
simplified error indicator. We also define
\begin{gather}
  ei^{ZZ} = \frac{\eta^{ZZ}}{ || \nabla e||_{L^2(\Omega)}}, \\
  ei^{A } = \frac{\eta^{A }}{ (\int_\Omega A\nabla e \cdot \nabla
    e + \eps (1-\eps) \int_\Omega A_\parallel \nabla_\parallel e_q
    \cdot \nabla_\parallel e_q)^{1/2}}, \\ 
  ei^{SA} = \frac{\eta^{SA}}{ (\int_\Omega A\nabla e \cdot \nabla e)^{1/2}}, 
\label{eq:Jhbb}
\end{gather}
the effectivity indices.

We test the robustness of the error indicators and the convergence of
the stabilized AP scheme in the following test
case. Let $\Omega = (0,1) \times (0,1)$, the anisotropy direction is
given by
\begin{gather}
  b = \frac{B}{|B|}\, , \quad
  B =
  \left(
    \begin{array}{c}
      \alpha  (2y-1) \cos (\pi x) + \pi \\
      \pi \alpha  (y^2-y) \sin (\pi x)
    \end{array}
  \right)
  \label{eq:J99a}\,\quad. 
\end{gather}
Note that we have $B \neq 0$ in the
computational domain. The parameter $\alpha$ describes the variations
of the anisotropy direction. For $\alpha = 0$ the anisotropy is
aligned in the direction of $x$ coordinate. We set $A_\perp =
A_\parallel = 1$. Now, we choose $\phi^\varepsilon $ to be a function
that converges to the limit solution $\phi^{0}$ as $\varepsilon
\rightarrow 0$:
\begin{gather}
  \phi^{0} = \sin \left(\pi y +\alpha (y^2-y)\cos (\pi x) \right), \\
  \phi^{\varepsilon } = \sin \left(\pi y +\alpha (y^2-y)\cos (\pi
  x) \right) + \varepsilon \cos \left( 2\pi x\right) \sin \left(\pi
  y \right)
  \label{eq:Jc0a}.
\end{gather}
Finally, the force term is calculated accordingly, i.e.
\begin{gather}
  f = - \nabla_\perp \cdot (A_\perp \nabla_\perp \phi^{\varepsilon })
  - \frac{1}{\varepsilon }\nabla_\parallel \cdot (A_\parallel \nabla_\parallel \phi^{\varepsilon })
  \nonumber.
\end{gather}

We study the effectivity indices on the unstructured meshes for
constant and variable anisotropy direction ($\alpha =0$ and $\alpha
=2$ respectively) and for small and large anisotropy ($\eps=1$ and
$\eps = 10^{-10}$ respectively).

\begin{table}[!ht]
  \centering
  \begin{tabular}{|c||c|c|c|c|}
    \hline\rule{0pt}{2.5ex}
    $h_1$--$h_2$ & $ei^{ZZ}$ & $ei^{A}$ & $ei^{SA}$ & $|| \nabla (\phi
    ^\eps_h -\phi ^\eps) ||_{L^2(\Omega )} / || \nabla \phi ^\eps_h ||_{L^2(\Omega )}$ \\
    \hline
    \hline $0.1-0.1$           & 1.05 & 2.53 & 2.53 & $1.5 \times 10^{-1}$\\
    \hline $0.05-0.05$         & 1.02 & 2.54 & 2.54 & $7.7 \times 10^{-2}$\\
    \hline $0.025-0.025$       & 1.01 & 2.54 & 2.54 & $3.9 \times 10^{-2}$\\
    \hline $0.0125-0.0125$     & 1.00 & 2.53 & 2.53 & $1.9 \times 10^{-2}$\\
    \hline $0.00625-0.00625$   & 1.00 & 2.53 & 2.53 & $9.8 \times 10^{-3}$\\
    \hline
  \end{tabular}
  \vspace{1ex}
  \\$\alpha = 0$, $\eps = 1$\\
  \vspace{2ex}
  \begin{tabular}{|c||c|c|c|c|}
    \hline\rule{0pt}{2.5ex}
    $h_1$--$h_2$ & $ei^{ZZ}$ & $ei^{A}$ & $ei^{SA}$ & $|| \nabla (\phi
    ^\eps_h -\phi ^\eps) ||_{L^2(\Omega )} / || \nabla \phi ^\eps_h ||_{L^2(\Omega )}$ \\
    \hline
    \hline $0.1-0.1$           & 0.99 & 4.74 & 4.68 & $8.0 \times 10^{-2}$\\
    \hline $0.05-0.05$         & 0.99 & 4.78 & 4.71 & $4.1 \times 10^{-2}$ \\
    \hline $0.025-0.025$       & 0.96 & 4.76 & 4.67 & $2.1 \times 10^{-2}$\\
    \hline $0.0125-0.0125$     & 0.93 & 4.89 & 4.65 & $1.1 \times 10^{-2}$\\
    \hline $0.00625-0.00625$   & 0.87 & 5.08 & 4.68 & $6.1 \times 10^{-3}$\\
    \hline
  \end{tabular}
  \vspace{1ex}
  \\$\alpha = 0$, $\eps = 10^{-10}$\\
  \vspace{2ex}
  \begin{tabular}{|c||c|c|c|c|}
    \hline\rule{0pt}{2.5ex}
    $h_1$--$h_2$ & $ei^{ZZ}$ & $ei^{A}$ & $ei^{SA}$ & $|| \nabla (\phi
    ^\eps_h -\phi ^\eps) ||_{L^2(\Omega )} / || \nabla \phi ^\eps_h ||_{L^2(\Omega )}$ \\
    \hline
    \hline $0.1-0.1$           & 1.05 & 2.54 & 2.54 & $1.5 \times 10^{-1}$\\
    \hline $0.05-0.05$         & 1.02 & 2.54 & 2.54 & $7.7 \times 10^{-2}$\\
    \hline $0.025-0.025$       & 1.01 & 2.54 & 2.54 & $3.9 \times 10^{-2}$\\
    \hline $0.0125-0.0125$     & 1.00 & 2.53 & 2.53 & $1.9 \times 10^{-2}$\\
    \hline $0.00625-0.00625$   & 1.00 & 2.53 & 2.53 & $9.9 \times 10^{-3}$\\
    \hline
  \end{tabular}
  \vspace{1ex}
  \\$\alpha = 2$, $\eps = 1$\\
  \vspace{2ex}
  \begin{tabular}{|c||c|c|c|c|}
    \hline\rule{0pt}{2.5ex}
    $h_1$--$h_2$ & $ei^{ZZ}$ & $ei^{A}$ & $ei^{SA}$ & $|| \nabla (\phi
    ^\eps_h -\phi ^\eps) ||_{L^2(\Omega )} / || \nabla \phi ^\eps_h ||_{L^2(\Omega )}$ \\
    \hline
    \hline $0.1-0.1$           & 0.99 & 4.07 & 3.99 & $ 1.1 \times 10^{-1}$\\
    \hline $0.05-0.05$         & 0.98 & 4.24 & 4.07 & $ 5.4 \times 10^{-2}$\\
    \hline $0.025-0.025$       & 0.97 & 4.30 & 4.09 & $ 2.7 \times 10^{-2}$\\
    \hline $0.0125-0.0125$     & 0.94 & 4.41 & 4.08 & $ 1.4 \times 10^{-2}$\\
    \hline $0.00625-0.00625$   & 0.90 & 4.69 & 4.19 & $ 7.5 \times 10^{-3}$\\
    \hline
  \end{tabular}
  \vspace{1ex}
  \\$\alpha = 2$, $\eps = 10^{-10}$
  
  \caption{Effectivity indices for isotropic meshes}
  \label{tab:ei_iso}
\end{table}

Table \ref{tab:ei_iso} shows the numerical results for isotropic
unstructured meshes in different regimes. In the case of no anisotropy
($\eps =1$) the Zienkiewicz-Zhu error estimator converges to true
error as $h$ goes to zero. The simplified and full effectivity indexes
are the same and converge also to a constant value. In the case of
small anisotropy ($\eps = 10^{-10}$) the effectivity index for
Zienkiewicz-Zhu error estimator is close to one for all testes
isotropic meshes in the case of variable anisotropy
direction. However, its value seems to decrease with the mesh size
meaning that the estimator slightly underestimate the true error for
fine meshes.  
The divergence
is observed for a constant direction of anisotropy and small value of
$\eps$.
This shows that the Zienkiewicz-Zhu error indicator is not always
equivalent to the true error.
The stabilized Asymptotic Preserving scheme converges to the exact
solution in all four cases with the optimal convergence rate.
%% In this case
%% however the 2D test case is in fact reduced to 1D case with exact
%% solution varying only in the $Y$ direction, i.e. the direction
%% perpendicular to the anisotropy. In this setting the isotropic mesh
%% refinement means that the mesh is refined also in the ``wrong''
%% direction. 

Table \ref{tab:ei_aniso_bc} presents the numerical results
corresponding to the of large anisotropy aligned with the coordinate
system. This time we are interested in the behavior of the error
indicators when the mesh refinement is anisotropic. In the first table
the mesh is refined in the direction perpendicular to the anisotropy
direction with aspect ration ranging from 10 to 1280. In this case the
Zienkiewicz-Zhu remains constant and close to 1.  The relative error
converges until the aspect ratio of 80 is reached. The effectivity
index for the full error indicator increases from $6.28$ to $15.6$
with the mesh size until the aspect ratio reaches the value of 160. At
the same time the effectivity index for the simplified error indicator
is between $5.57$ and $6.64$. This suggests that the latter could
perform better in the anisotropic mesh refinement. Its effectivity
index does not seem to depend on the aspect ration wen the mesh is
refined in the ``right'' direction (perpendicular to the anisotropy).

Next, the
influence of the mesh refinement in the ``wrong'' (parallel to the
anisotropy) direction is performed. For aspect ration ranging from 1
to 16 the divergence of the $ei^{ZZ}$ and the relative error is
clearly observed. In fact, all effectivity indexes approach zero with
the refinement.
The last table displays the results of the
convergence of $ei^{ZZ}$ in the case of anisotropic mesh with aspect
ration 4 and triangles aligned in the ``wrong'' direction. In this
case, when the mesh is refined in both direction, the effectivity
index for Zienkiewicz-Zhu error estimator approaches 1. The
effectivity indexes of both error indicator diverge.

\begin{table}[!ht]
  \centering
  \begin{tabular}{|c||c|c|c|c|}
    \hline\rule{0pt}{2.5ex}
    $h_1$--$h_2$ & $ei^{ZZ}$ & $ei^{A}$ & $ei^{SA}$ & $|| \nabla (\phi
    ^\eps_h -\phi ^\eps) ||_{L^2(\Omega )} / || \nabla \phi ^\eps_h ||_{L^2(\Omega )}$ \\
    \hline
    \hline $0.1-0.01$         & 0.98 & 6.28 & 5.83 & $7.8 \times 10^{-3}$\\
    \hline $0.1-0.005$        & 0.97 & 7.68 & 6.09 & $4.3 \times 10^{-3}$ \\
    \hline $0.1-0.0025$       & 0.95 & 10.6 & 6.20 & $2.5 \times 10^{-3}$\\
    \hline $0.1-0.00125$      & 0.93 & 14.2 & 6.64 & $1.7 \times 10^{-3}$\\
    \hline $0.1-0.000625$     & 0.95 & 15.6 & 5.57 & $1.6 \times 10^{-3}$\\
    \hline $0.1-0.0003125$    & 0.98 & 9.88 & 3.94 & $2.3 \times 10^{-3}$\\
    \hline $0.1-0.00015625$   & 0.97 & 9.20 & 4.01 & $2.2 \times 10^{-3}$\\
    \hline $0.1-0.000078125$  & 0.98 & 5.09 & 2.95 & $3.1 \times 10^{-3}$\\
    \hline
  \end{tabular}
  \vspace{1ex}
  \\Aspect ratio from 1:10 to 1:1280\\
  \vspace{2ex}
  \begin{tabular}{|c||c|c|c|c|}
    \hline\rule{0pt}{2.5ex}
    $h_1$--$h_2$ & $ei^{ZZ}$ & $ei^{A}$ & $ei^{SA}$ & $|| \nabla (\phi
    ^\eps_h -\phi ^\eps) ||_{L^2(\Omega )} / || \nabla \phi ^\eps_h ||_{L^2(\Omega )}$ \\
    \hline
    \hline $0.1 -0.1$         & 0.99 & 4.74 & 4.68 & $8.0 \times 10^{-2}$\\
    \hline $0.05-0.1$         & 0.91 & 4.56 & 4.33 & $1.2 \times 10^{-1}$\\
    \hline $0.025-0.1$        & 0.32 & 4.75 & 4.04 & $4.4 \times 10^{-1}$\\
    \hline $0.0125-0.1$       & 0.005 & 0.44 & 0.36 & $5.2 \times 10^{1}$\\
    \hline $0.00625-0.1$      & 0.0002 & 0.075 & 0.059 & $1.9 \times 10^{3}$\\
    %% \hline $0.003125-0.1$     & 0.66 & 1.63 & 1.60 & 3.1 \times 10^{3}\\
    \hline
  \end{tabular}
  \vspace{1ex}
  \\Aspect ratio from 1:1 to 16:1\\
  \vspace{2ex}

  \begin{tabular}{|c||c|c|c|c|}
    \hline\rule{0pt}{2.5ex}
    $h_1$--$h_2$ & $ei^{ZZ}$ & $ei^{A}$ & $ei^{SA}$ & $|| \nabla (\phi
    ^\eps_h -\phi ^\eps) ||_{L^2(\Omega )} / || \nabla \phi ^\eps_h ||_{L^2(\Omega )}$\\
    \hline
    \hline $0.025-0.1$         & 0.32 & 4.75 & 4.04 & $4.4 \times 10^{-1}$\\
    \hline $0.0125-0.05$       & 0.40 & 6.66 & 5.66 & $2.0 \times 10^{-1}$\\
    \hline $0.00625-0.025$     & 0.53 & 8.86 & 7.53 & $7.4 \times 10^{-2}$\\
    \hline $0.003125-0.0125$   & 0.55 & 9.54 & 8.09 & $3.6 \times 10^{-2}$\\
    \hline $0.0015625-0.00625$ & 0.69 & 12.87 & 10.09 & $1.4 \times 10^{-2}$\\
    \hline
  \end{tabular}
  \vspace{1ex}
  \\Aspect ratio 4:1 

  \caption{Effectivity indices for anisotropic meshes for $\alpha = 0$
    and $\eps=10^{-10}$}
  \label{tab:ei_aniso_bc}
\end{table}

\subsubsection{Mesh adaptation}

We now apply our adaptive algorithm to build a sequence triangulations
in the following way starting from an isotropic unstructured grid with
$h=0.02$. At every iteration of the algorithm the error indicator is
used to construct a subsequent mesh. We compare results of the
simplified and full error indicators in various regimes: small and
large anisotropy, $b$ direction constant and variable. We focus on the
resulting mesh size and error in the $H^1$-norm as well as on the
error convergence in terms of prescribed tolerance $TOL$.

Let $\Omega = (0,1) \times (0,1)$, the anisotropy direction is given
by (\ref{eq:J99a}). We set $A_\perp = A_\parallel = 1$. We choose
$\phi^\varepsilon $ to be a function that converges to the limit
solution $\phi^{0}$ as $\varepsilon \rightarrow 0$:
\begin{gather}
  \phi^{0} = \sin \left(\pi y +\alpha (y^2-y)\cos (\pi x) \right)
  e^{-(\frac{\pi y +\alpha (y^2-y)\cos (\pi x) - 0.5}{\delta})^2}
  , \\
  \phi^{\varepsilon } = \sin \left(\pi y +\alpha (y^2-y)\cos (\pi
  x) \right) 
  e^{-(\frac{\pi y +\alpha (y^2-y)\cos (\pi x) - 0.5}{\delta})^2}
  + \varepsilon \cos \left( 2\pi x\right) \sin \left(\pi
  y \right)
  \label{eq:Jc0a}.
\end{gather}
Finally, the force term is calculated accordingly. The limit solution
is nothing else than the limit solution from previous section
multiplied by a Gaussian following the anisotropy direction. The
parameter $\delta$ controls the width of the exponential part. Setting
$\delta = 0.1$ in our simulations yields a solution which has a strong
gradient in the direction perpendicular to the anisotropy direction in
a small subregion of a computational domain. The adaptive algorithm
should be able to capture this strong variation of the solution and
produce a mesh that is much finer in this subregion than in the
remaining part of the domain.

$ $
\vspace{1ex}

\paragraph{Small anisotropy $\eps =1$, constant and variable direction of $b$ ($\alpha = 0$ and
  $\alpha = 2$)} 

$ $
\vspace{1ex}

In the first two test cases the adaptive algorithm is studied in the
$\eps =1$ regime, {\it i.e.} when no anisotropy is present. In this
case the two error indicators : full and simplified are equivalent.

\begin{table}
  \centering
  \begin{tabular}{|c||c|c|c|c|c|}
    \hline\rule{0pt}{2.5ex}
    $TOL$ & $err$ & $NV$ & $ei^{ZZ}$ & $ei^{A}$ & $ei^{SA}$ \\
    \hline
    \hline 0.25    & 0.096 &   698 & 1.03 & 2.55 & 2.55 \\
    \hline 0.125   & 0.048 &  2457 & 1.01 & 2.54 & 2.54 \\
    \hline 0.0625  & 0.024 &  8834 & 1.00 & 2.57 & 2.57 \\
    \hline 0.03125 & 0.012 & 34587 & 1.00 & 2.54 & 2.54 \\
    \hline
  \end{tabular}
  \caption{$H^1$ error ($err$), number of nodes ($NV$) and effectivity
    indices for mesh iteration $15$ for constant direction of $b$ and $\eps =1$}
  \label{tab:adapt_bc_e1}
\end{table}
\begin{table}
  \centering
  \begin{tabular}{|c||c|c|c|c|c|}
    \hline\rule{0pt}{2.5ex}
    $TOL$ & $err$ & $NV$ & $ei^{ZZ}$ & $ei^{A}$ & $ei^{SA}$ \\
    \hline
    \hline 0.25    & 0.094 &   785 & 1.03 & 2.58 & 2.58 \\
    \hline 0.125   & 0.047 &  2696 & 1.01 & 2.59 & 2.59 \\
    \hline 0.0625  & 0.024 & 10141 & 1.00 & 2.59 & 2.59 \\
    \hline 0.03125 & 0.012 & 39035 & 1.00 & 2.58 & 2.58 \\
    \hline
  \end{tabular}
  \caption{Relative $H^1$ error ($err$), number of nodes ($NV$) and effectivity
    indices for mesh iteration $15$ for variable direction of $b$ and $\eps =1$}
  \label{tab:adapt_bv_e1}
\end{table}

Tables \ref{tab:adapt_bc_e1} and \ref{tab:adapt_bv_e1} show the
results for $b$ field with constant and variable direction
respectively. The values in the tables are given after 15 iterations
of mesh adaptation algorithm. In both cases the optimal convergence is
obtained. The true error is clearly related to the prescribed error
tolerance $TOL$ and the node number is multiplied by 4 every time
$TOL$ is divided by 2. The Z-Z effectivity index converges to 1 with
$TOL$ and the values of indexes for error indicators remain almost
constant. This is not surprising since in this case the proposed error
indicators reduce to the standard {\it a posteriori} error indicator
studied before. The adapted meshes are
presented on Figure \ref{fig:mesh_eps1}. 

Numerical relative error obtained on the isotropic uniform mesh with
$h=0.00625$ (31325 mesh points) give the relative error equal to
$0.021$, which is comparable with the results obtained for
$TOL=0.0625$. The adapted giving the same numerical precision are
three times smaller.

\def\xxxa{0.46\textwidth}
\begin{figure}[!ht]
  \centering
  \subfigure[solution for constant $b$]
  {\includegraphics[angle=0,width=\xxxa]{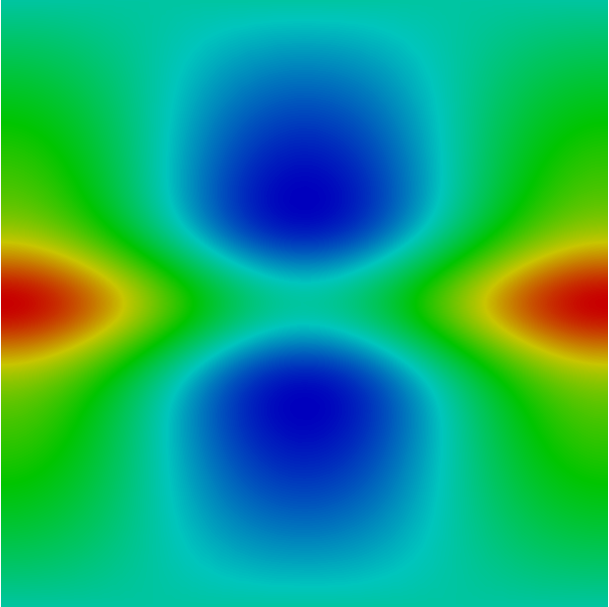}}
  \subfigure[adapted mesh]
  {\includegraphics[angle=0,width=\xxxa]{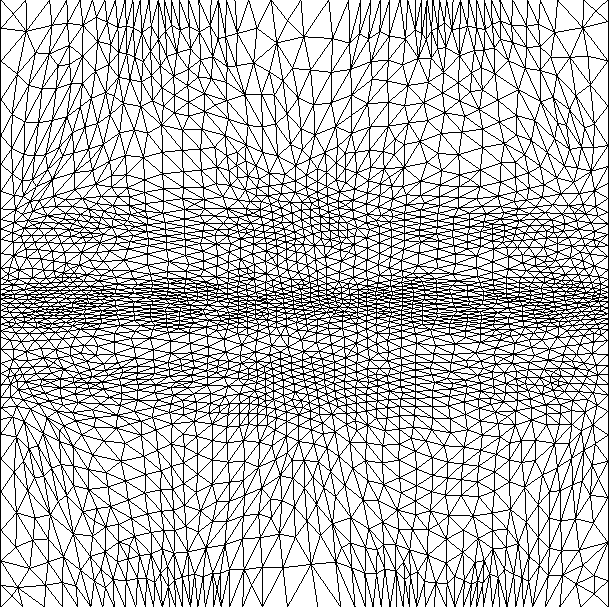}}
\\
  \subfigure[solution for variable $b$]
  {\includegraphics[angle=0,width=\xxxa]{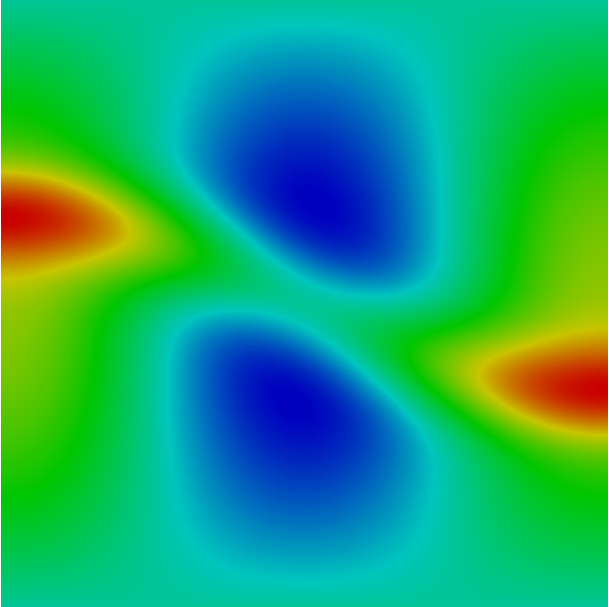}}
  \subfigure[adapted mesh]
  {\includegraphics[angle=0,width=\xxxa]{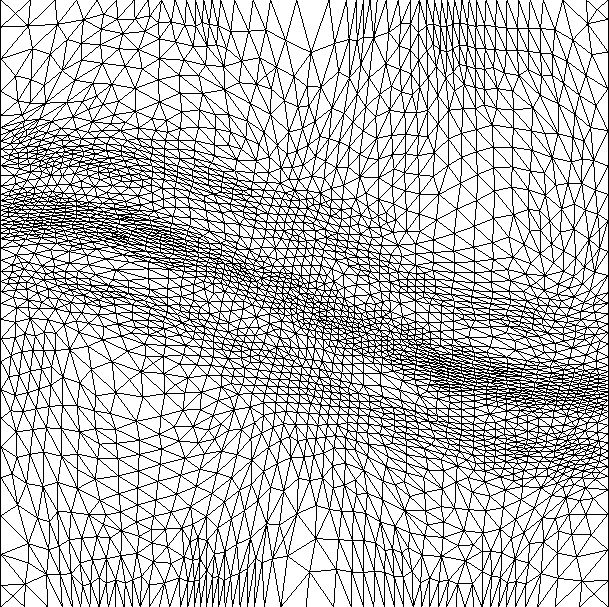}}

  \caption{Exact solution and meshes obtained after 15 iterations for
    $TOL=0.125$ with $\eps=1$  for constant and variable direction of $b$.}
  \label{fig:mesh_eps1}
\end{figure}

$ $
\vspace{1ex}

\paragraph{constant direction of $b$ ($\alpha = 0$), large anisotropy $\eps =10^{-10}$}

$ $
\vspace{1ex}

In the next test case we consider large anisotropy $\eps = 10^{-10}$
and aligned $b$ direction. The simplified error indicator and the full
error indicator are no longer equivalent. The results presented in
Table \ref{tab:adapt_bc_e-10} display the true error and effectivity
indexes obtained by applying those two different algorithms. In this
particular case we display results after 30 mesh adaptations. The
number is bigger than in previous case in order to allow the algorithm
to fully converge and exploit the reduced dimensionality of this
particular test. Note that in both cases the true error is comparable
and converges with $TOL$. The Zienkiewicz-Zhu effectivity index is
close to 1 for both error indicator. The aspect ratio for the smallest
$TOL$ studied is over 500. The simplified error indicator seems to
perform better : the mesh size for the smallest $TOL$ tested is three
times smaller than for the full error indicator. The relative $H^1$
error is also slightly smaller for the simplified error
indicator. 
%% Neglecting the $\tilde{G}_K(q^\eps _h)$ term leads to more
%% anisotropic meshes with the average aspect ration greater than 100 for
%% the smallest error tolerance tested. This is almost five times more
%% than for the full error indicator. Apparently the numerical precision
%% does not suffer from this simplification. 
The adapted meshes are presented on Figure
\ref{fig:mesh_bcst_eps1e-10}.

Numerical relative error obtained on the isotropic uniform mesh with
$h=0.00625$ (31325 mesh points) give the relative error equal to
$0.035$, which is comparable with the results obtained for
$TOL=0.0625$. The adapted giving the same numerical precision are
115 (40) times smaller for the simplified (full) error indicator.

\def\xxxa{0.31\textwidth}
\begin{figure}[!ht]
  \centering
  \subfigure[solution]
  {\includegraphics[angle=0,width=\xxxa]{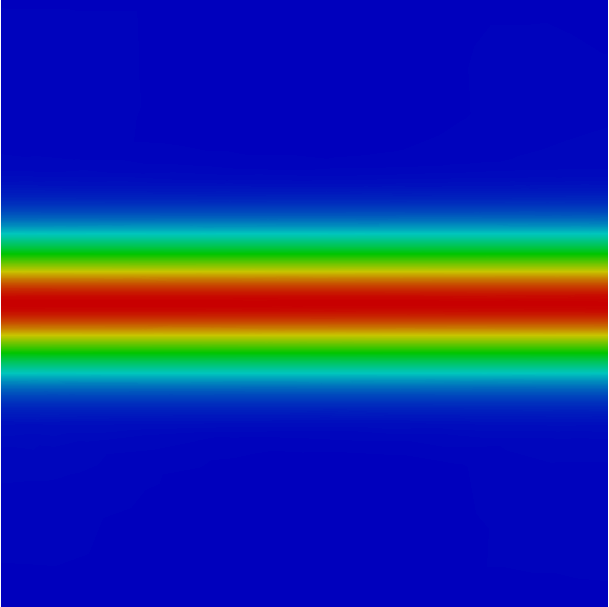}}
  \subfigure[full error indicator]
  {\includegraphics[angle=0,width=\xxxa]{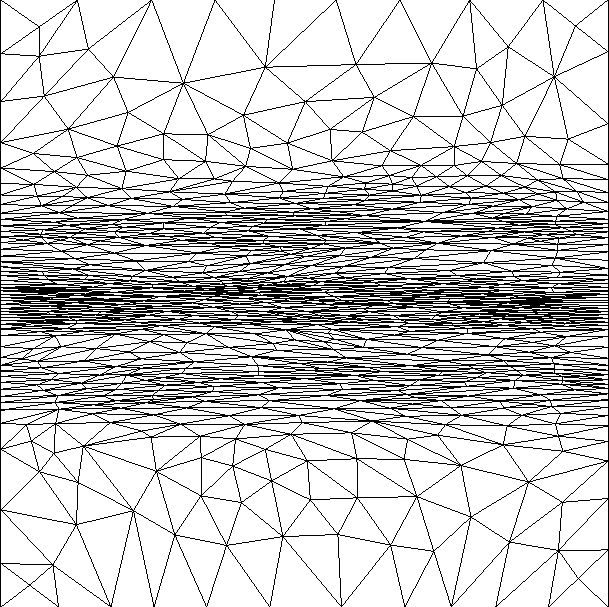}}
  \subfigure[simplified error indicator]
  {\includegraphics[angle=0,width=\xxxa]{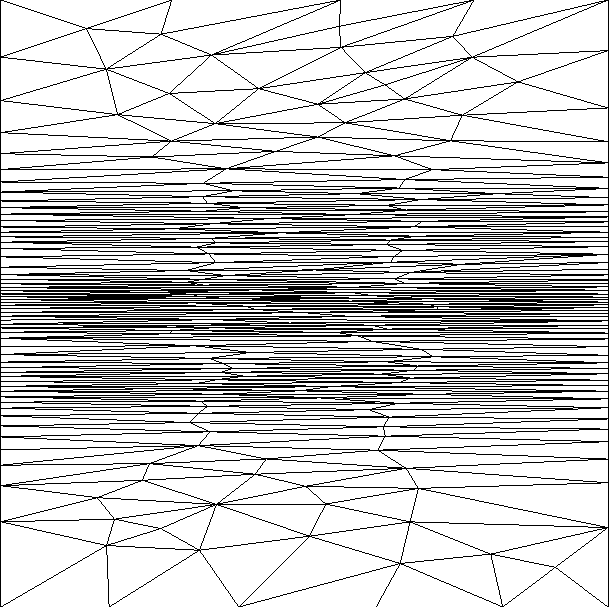}}

  \caption{Exact solution and meshes obtained after 30 iterations for
    $TOL=0.125$ with $\eps=10^{-10}$  and constant direction of $b$.}
  \label{fig:mesh_bcst_eps1e-10}
\end{figure}

\begin{table}
  \centering
  \begin{tabular}{|c||c|c|c|c|c|c|}
    \hline\rule{0pt}{2.5ex}
    $TOL$ & $err$ & $NV$ & $(\frac{h_2}{h_1})_{max}$& $(\frac{h_2}{h_1})_{avg}$ & $ei^{ZZ}$ & $ei^{A}$ \\
    \hline
    \hline 0.25    & 0.072 &   272 &  67 &  12 & 1.01 & 3.20 \\
    \hline 0.125   & 0.037 &   758 &  86 &  14 & 1.01 & 3.26 \\
    \hline 0.0625  & 0.018 &  2435 &  91 &  17 & 0.99 & 3.32 \\
    \hline 0.03125 & 0.0093 &  6642 & 296 & 23 & 0.98 & 3.28 \\
    \hline
  \end{tabular}
  \vspace{1ex}
  \\full error indicator\\
  \vspace{2ex}

  \begin{tabular}{|c||c|c|c|c|c|c|}
    \hline\rule{0pt}{2.5ex}
    $TOL$ & $err$ & $NV$ & $(\frac{h_2}{h_1})_{max}$& $(\frac{h_2}{h_1})_{avg}$ & $ei^{ZZ}$ & $ei^{SA}$ \\
    \hline
    \hline 0.25    & 0.060 &   105 &  130 &  35 & 1.01 & 3.65 \\
    \hline 0.125   & 0.031 &   271 &  224 &  48 & 1.00 & 3.49 \\
    \hline 0.0625  & 0.016 &   652 &  501 &  87 & 0.98 & 3.74 \\
    \hline 0.03125 & 0.0076 & 2018 & 536 & 106 & 0.99 & 4.07 \\
    \hline
  \end{tabular}
  \vspace{1ex}
  \\simplified error indicator\\

  \caption{Relative $H^1$ error ($err$), number of nodes ($NV$), maximum aspect
    ration ($(\frac{h_2}{h_1})_{max}$), average aspect ration
    ($(\frac{h_2}{h_1})_{avg}$) and effectivity indices for mesh
    iteration $30$ for constant direction of $b$ and $\eps =10^{-10}$}
  \label{tab:adapt_bc_e-10}
\end{table}

$ $
\vspace{1ex}

\paragraph{variable direction of $b$ ($\alpha = 2$), large anisotropy $\eps =10^{-10}$}
$ $
\vspace{1ex}

In the last studied test case we have applied the mesh adaptation
algorithm to the problem with large anisotropy with variable
direction. Table \ref{tab:adapt_bv_e1} shows obtained results of
numerical simulations. The simplified error indicator performs more
efficiently than the full error indicator. Poor performance of the
full error indicator for the smallest tolerance is caused by the
perpendicular derivatives of $q^\eps_h$ which cause the over
refinement in the direction perpendicular to the anisotropy
direction. The resulting mesh is almost eight times bigger. For
smaller values of the tolerance the difference in mesh sizes is much
smaller and the meshes constructed for the full error indicator give
slightly better precision. In both cases the Z-Z error estimator is
close to 1. The adapted meshes are
presented on Figure \ref{fig:mesh_bvar_eps1e-10}.

Numerical relative error obtained on the isotropic uniform mesh with
$h=0.00625$ (31325 mesh points) give the relative error equal to
$0.038$, which is comparable with the results obtained for
$TOL=0.0625$. The adapted giving the same numerical precision are
20 (10) times smaller for the simplified (full) error indicator.

\def\xxxa{0.31\textwidth}
\begin{figure}[!ht]
  \centering
  \subfigure[solution]
  {\includegraphics[angle=0,width=\xxxa]{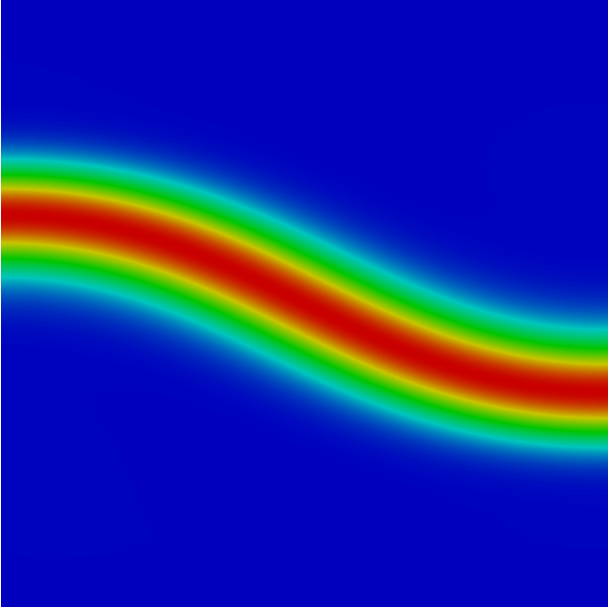}}
  \subfigure[full error indicator]
  {\includegraphics[angle=0,width=\xxxa]{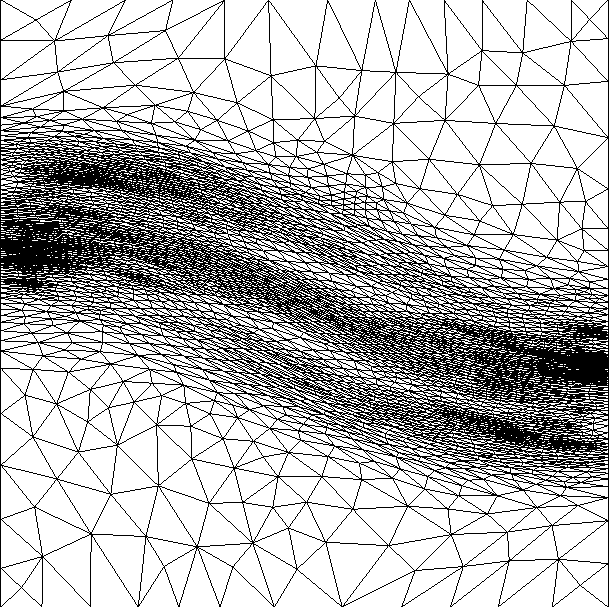}}
  \subfigure[simplified error indicator]
  {\includegraphics[angle=0,width=\xxxa]{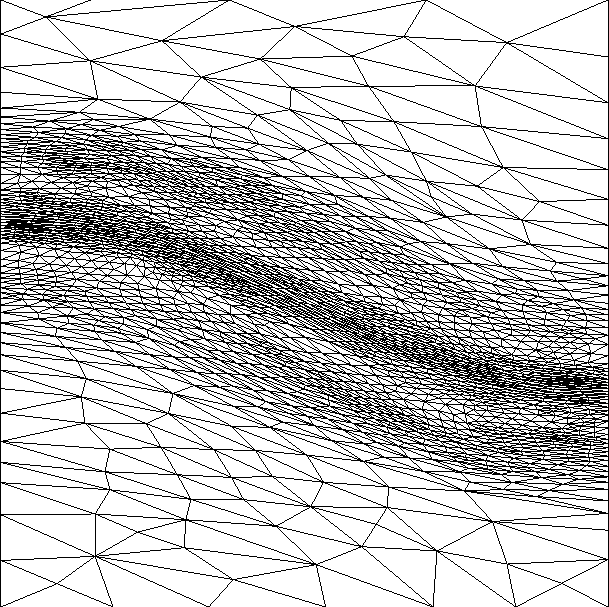}}

  \caption{Exact solution and meshes obtained after 30 iterations for
    $TOL=0.125$ with $\eps=10^{-10}$  and variable direction of $b$.}
  \label{fig:mesh_bvar_eps1e-10}
\end{figure}

\begin{table}
  \centering
  \begin{tabular}{|c||c|c|c|c|c|c|}
    \hline\rule{0pt}{2.5ex}
    $TOL$ & $err$ & $NV$ & $(\frac{h_2}{h_1})_{max}$& $(\frac{h_2}{h_1})_{avg}$ & $ei^{ZZ}$ & $ei^{A}$ \\
    \hline
    \hline 0.5    & 0.137 &   183 &  16 & 5.4 & 1.04 &  3.21 \\
    \hline 0.25   & 0.070 &   587 &  21 & 5.8 & 1.01 &  3.45 \\
    \hline 0.125  & 0.033 &  3195 &  54 & 8.2 & 0.99 &  3.83 \\
    \hline 0.0625 & 0.015 & 52658 & 165 &  17 & 0.98 &  4.88 \\
    \hline
  \end{tabular}
  \vspace{1ex}
  \\full error indicator\\
  \vspace{2ex}

  \begin{tabular}{|c||c|c|c|c|c|c|}
    \hline\rule{0pt}{2.5ex}
    $TOL$ & $err$ & $NV$ & $(\frac{h_2}{h_1})_{max}$& $(\frac{h_2}{h_1})_{avg}$ & $ei^{ZZ}$ & $ei^{A}$ \\
    \hline
    \hline 0.5     & 0.15  &   138 & 27 & 5.98 & 1.03 & 3.18 \\
    \hline 0.25    & 0.073 &   445 & 25 & 6.88 & 1.01 & 3.34 \\
    \hline 0.125   & 0.037 &  1720 & 33 & 7.07 & 1.00 & 3.29 \\
    \hline 0.0625  & 0.018 &  6884 & 43 & 7.48 & 0.97 & 3.36 \\
    \hline
  \end{tabular}
  \vspace{1ex}
  \\simplified error indicator\\

  \caption{Relative $H^1$ error ($err$), number of nodes ($NV$), maximum aspect
    ration ($(\frac{h_2}{h_1})_{max}$), average aspect ration
    ($(\frac{h_2}{h_1})_{avg}$) and effectivity indices for mesh
    iteration $15$ for variable direction of $b$ and $\eps =10^{-10}$}
  \label{tab:adapt_bv_e1}
\end{table}

%%%%%%%%%%%%%%%%%%%%%%%%%%%%
\section{Conclusion}
%%%%%%%%%%%%%%%%%%%%%%%%%%%%

A stabilized Asymptotic Preserving method for strongly anisotropic
Laplace equation has been proposed and tested numerically. The error
indicators including first order derivatives has been developed for
this reformulated problem. Numerical experiments show the performance
of the remeshing routine. The resulting meshes are considerably
smaller by the factor from 3 to 115 than the isotropic uniform grids
giving the same precision. The biggest gain is obtained for strong
anisotropy in the constant direction.

\bibliographystyle{abbrv}
\bibliography{bib_aniso}

\end{document}